\documentclass[11pt]{amsart}
\usepackage{xcolor}
\usepackage{mathtools}
\usepackage[latin2]{inputenc}
\usepackage{amsmath, amssymb}

\usepackage[
  hmarginratio={1:1},     
  vmarginratio={1:1},     
  textwidth=430pt,        
  heightrounded,          
]{geometry}

\usepackage{hyperref}
\usepackage{enumerate}
\usepackage{tikz}
\usepackage{graphicx} 
\usepackage{subcaption}

\mathtoolsset{showonlyrefs,showmanualtags}

\newtheorem{theorem}{Theorem}[section]
\newtheorem{lemma}[theorem]{Lemma}
\newtheorem{proposition}[theorem]{Proposition}

\theoremstyle{definition}
\newtheorem{definition}[theorem]{Definition}
\theoremstyle{remark}
\newtheorem{remark}[theorem]{Remark}

\newcommand{\R}{\mathbf{R}}
\def\Xint#1{\mathchoice
{\XXint\displaystyle\textstyle{#1}}%
{\XXint\textstyle\scriptstyle{#1}}%
{\XXint\scriptstyle\scriptscriptstyle{#1}}%
{\XXint\scriptscriptstyle\scriptscriptstyle{#1}}%
\!\int}
\def\XXint#1#2#3{{\setbox0=\hbox{$#1{#2#3}{\int}$ }
\vcenter{\hbox{$#2#3$ }}\kern-.6\wd0}}

\def\dashint{\Xint-}

\allowdisplaybreaks

\title[Tumor growth with nutrients]{Tumor growth with nutrients: stability of the tumor patches}

\author{Inwon Kim}
\address[Inwom Kim]{University of California, Los Angeles, Box 951555, Los Angeles, CA 90095, USA}
\email{ikim@math.ucla.edu}

\author{Jona Lelmi}
\address[Jona Lelmi]{Institut f\"ur angewandte Mathematik, Universit\"at Bonn, Endenicher Allee 60, 53115 Bonn, Germany}
\email{lelmi@hcm.uni-bonn.de}

\begin{document}

\maketitle

\begin{abstract}
In this paper, we study a tumor growth model with nutrients. The contact inhibition for the tumor cells, presented in the model, results in the evolution of a congested {\it tumor patch}. We study the regularity of the tumor patch as the nutrients' diffusion strength $D$ diminishes. In particular, we show that for small $D>0$ the boundary of the tumor patch stays in a small neighborhood of the smooth  tumor patch boundary obtained with $D=0$, uniformly with respect to the Hausdorff distance.

 \medskip


  \medskip

\noindent \textbf{Mathematical Subject Classification (MSC2020)}: 35B65 (primary), 35Q92.
	
\end{abstract}

\section{Introduction}\label{sec:intro}


We consider a basic model that describes tumor growth with nutrients, given by the following system of PDEs, set in $Q:= \R^d\times [0,\infty)$:
\begin{equation}\label{eq:model_eqnts}
\begin{cases}
\partial_t \rho_t - \operatorname{div}\left(\rho\nabla p \right) &= (n-b)\rho,\quad \rho \le 1,\quad p \in P_{\infty}(\rho),
\\ \partial_t n - D\Delta n &= -\rho n,\quad n \to c > 0\ \text{as}\ |x|\to \infty,
\end{cases}
\end{equation}
Here $D, b, c$  are nonnegative constants, and $P_{\infty}$ denotes the Hele-Shaw graph, given by
\begin{align*}
P_{\infty}(\rho) = \begin{cases}
0\ &\text{if}\ \rho < 1,
\\ [0, +\infty)\ &\text{if}\ \rho = 1.
\end{cases}
\end{align*}
 This system has been actively studied in recent literature: see for instance \cite{Maury2014, PQV, DP21, Jacobs2022, GKM}. Here $\rho$ and $n$ each denote the density of tumor cells and the nutrients.  The tumor cells grow by the nutrients which are supplied by the external environment, while also dying at rate $b$. The condition $p\in P_{\infty}(\rho)$ implies that the pressure variable $p$ is the Lagrange multiplier for the constraint $\rho \le 1$, which represents the contact inhibition in cells.  In the region $n\geq b$, the cells only grow within the constraints, so the solution features time-evolving {\it patches} of congested cell region $\{\rho=1\}$. In this paper we will only consider the model with no death ($b=0$), and only solutions with initial density given as a characteristic function $\chi_{\Omega(0)}$, which then will evolve as a characteristic function $\rho(\cdot,t)= \chi_{\Omega(t)}$. In this paper, we consider the model with no death term ($b=0$). Moreover, for the Hausdorff convergence of the boundaries (Theorem~\ref{thm:conv_patches}), we consider only solutions with initial density given as a characteristic function $\chi_{\Omega(0)}$, which then will evolve as a characteristic function $\rho(\cdot,t)= \chi_{\Omega(t)}$. Our focus is on the regularity properties of the patch boundary $\partial\Omega(t)$.

 \medskip
 
While the system is well-posed, it presents rather curious instability of patch boundaries. When $D > 0$, the tumor patch appears to generate growing fingers, as observed by numerical experiments \cite{kitsu97, Maury2014, PTV14}. (Also see the recent instability analysis in \cite{TZ} for a closely related model). See Figure~\ref{fig:image1} for a numerical simulation. {(Figure~\ref{fig:image1} and Figure~\ref{fig:image2} contain snapshots of simulations run by Wonjun Lee. The time evolution is from left to right, and the white set corresponds to the tumor patch. The implementation is based on the numerical scheme in \cite{JLL}, which is consistent with the scheme we discuss in this paper.)}

\begin{figure}[ht] 
   \begin{subfigure}{0.32\textwidth}
       \includegraphics[width=\linewidth]{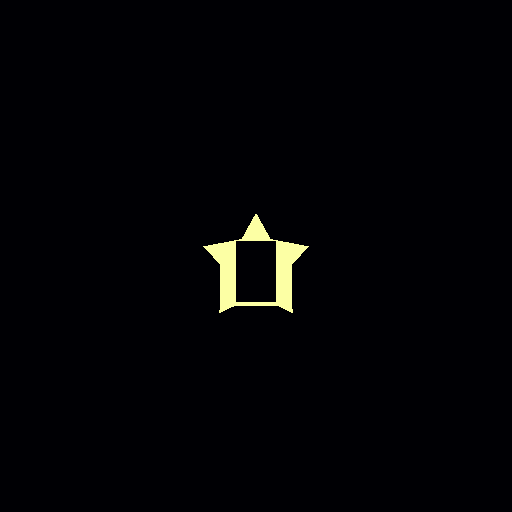}
   \end{subfigure}
\hfill 
   \begin{subfigure}{0.32\textwidth}
       \includegraphics[width=\linewidth]{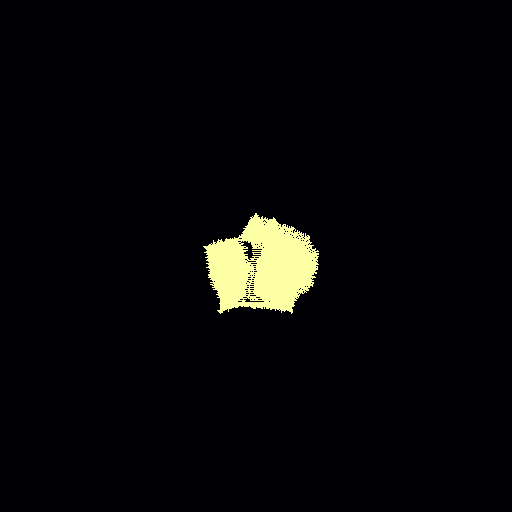}
   \end{subfigure}
\hfill 
   \begin{subfigure}{0.32\textwidth}
       \includegraphics[width=\linewidth]{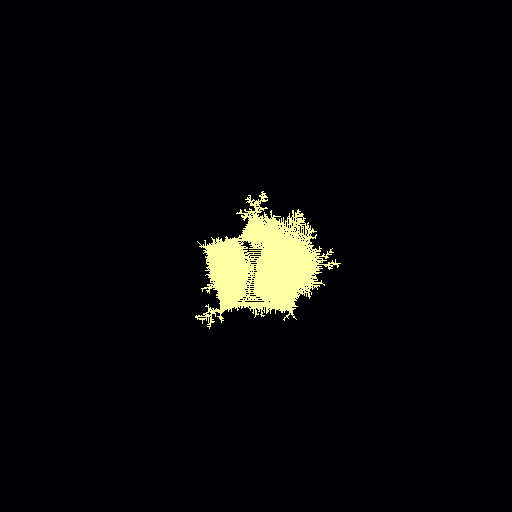}
   \end{subfigure}

   \caption{Numerical solution with constant initial nutrient, $b = 0$, $D~=~10^{-6}$. Thanks to Wonjun Lee for running the simulation.}
   \label{fig:image1}
\end{figure}

The growth of oscillation on the patch boundary appears to be caused by the competitive (prey-predator) nature of the density and the nutrient. One could further conjecture that the irregularity may get worse as $D$ decreases since less diffusion would make the nutrients less averaged and irregular. It is thus surprising that the nutrient diffusion is indeed essential in this de-regularizing phenomenon: when $D=0$ it was shown in \cite{Jacobs2022} that the patch boundary becomes smooth in finite time when there is a sufficient amount of nutrients at the initial time. 

\medskip

In view of the above discussion, it is natural to ask whether some information is lost between the model with a small choice of $D$ and with $D=0$. In this paper we focus on this question: we show that within finite time the answer is no, at least in terms of Hausdorff distance. Namely, we will show that, for any given finite time interval,  the patch boundary for small $D$ lies in a small neighborhood of the smooth patch boundary for $D=0$. While our results do not rule out small-scale irregularities that vanish as $D$ tends to zero, we can at least say that there are no persisting irregularities at a fixed scale in the diffusion zero limit (see Theorem 1). We should emphasize though that this convergence only holds within a fixed finite time interval: we in fact conjecture that as $D$ tends to zero, the finger growth will slow down and only be significant after a long time. The role of nutrient diffusion in this exhibition of instability remains to be understood.

\medskip

There are many open questions that remain to be investigated. For instance, we suspect that a stronger mode of convergence holds true for the patch boundary. For example, we conjecture that the perimeters of the tumor patches converge as well in our setting, however, we are not able to verify it. The effect of the cell death rate ($b>0$) is also yet to be studied. With a positive death rate, the tumor patch develops a necrotic core, where the density falls below one. In this case, the fingers growth on the outer boundary of the set $\{\rho=1\}$ was observed even when $D=0$, see Figure~\ref{fig:image2}.

\begin{figure}[ht] 
   \begin{subfigure}{0.32\textwidth}
       \includegraphics[width=\linewidth]{barenblatt-0.png}
   \end{subfigure}
\hfill 
   \begin{subfigure}{0.32\textwidth}
       \includegraphics[width=\linewidth]{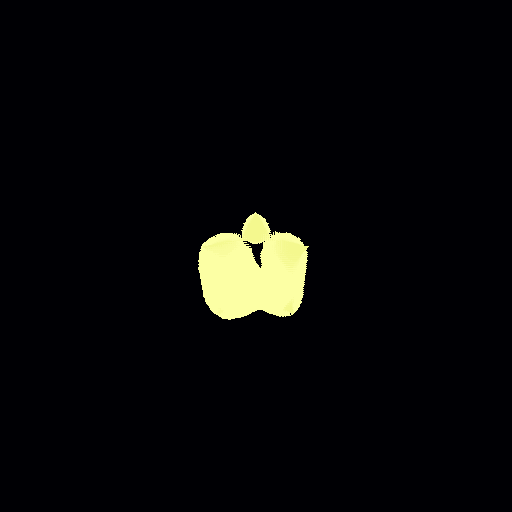}
   \end{subfigure}
\hfill 
   \begin{subfigure}{0.32\textwidth}
       \includegraphics[width=\linewidth]{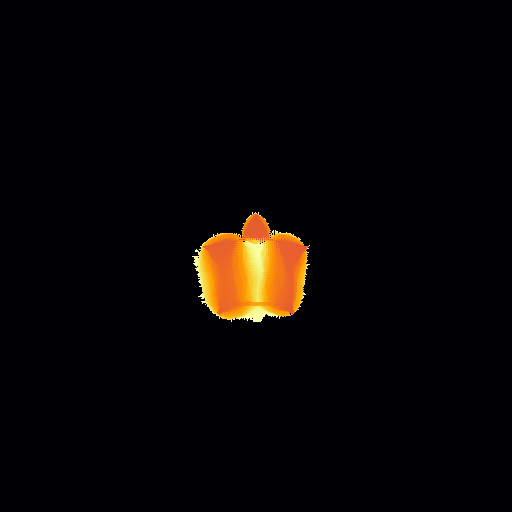}
   \end{subfigure}

   \caption{Numerical solution with constant initial nutrient, $b = 0.2$, $D~=~0$.}
   \label{fig:image2}
\end{figure}

  It would be interesting to understand if our analysis can be extended to $b>0$ when this parameter is assumed to converge to zero. In this case, the non-monotone nature of the patch evolution poses a challenge. On the other hand, the main technical result of our work,  namely the $H^1$ convergence of the pressure variable (Proposition \ref{prop:strong_conv_pressure}), may be suitably extended to this setting.  
 
 \medskip
 
 The main step in our analysis is improving the $L^1$-convergence of the tumor patches (proved in Theorem \ref{thm:conv_D_zero}) to the Hausdorff convergence of their boundaries. This is done by using both (i) the monotone-expanding property of the tumor patches, 
 and (ii) the strong $H^1$ convergence of the pressure variables, proved in Proposition \ref{prop:strong_conv_pressure}. The latter is based on a variational characterization of the pressure variable, following ideas similar to that of \cite{KIM2022}, where the authors prove it in a Keller-Segel model for chemotaxis.

 \medskip

 The rest of the paper is organized as follows: in Section \ref{sec:mainres} we state our main results, namely, Theorem \ref{thm:conv_patches} and Proposition \ref{prop:strong_conv_pressure}. In Section \ref{sec:prelim} we recall the definition of weak solution for the system \eqref{eq:model_eqnts}, and its corresponding existence and uniqueness. In Section \ref{sec:schemes} we recall the approximation scheme used to construct weak solutions to \eqref{eq:model_eqnts} and a slight variant of it that ensures a stronger convergence of the nutrient variable.  In Section \ref{sec:proofs} we present the proofs of our results.
 
 \section{Main results}\label{sec:mainres}
In this section, we state the main results of this paper. Here and in the rest of the paper we will always assume that there is no death term, i.e.\ we set $b = 0$ in \eqref{eq:model_eqnts}. 

\medskip

\textbf{Notation}. Given an open subset $\Omega \subset \mathbf{R}^d$ we denote by $\mathcal{M}(\Omega)$ the Banach space of Radon measures with finite total variation on $\Omega$, endowed with the total variation norm, which we denote by $\Vert \cdot \Vert_{\mathcal{M}(\Omega)}$. Given a map $\mu: 2^{\Omega} \to \mathbf{R}$, we say that $\mu \in \mathcal{M}_{loc}(\Omega)$ if $\mu|_{2^{\tilde{\Omega}}} \in \mathcal{M}(\tilde{\Omega})$ for every $\tilde{\Omega} \subset \subset \Omega$. We also warn the reader that for a map $\mu: 2^{\Omega} \to \mathbf{R}^d$ we will write $\mu \in \mathcal{M}(\Omega)$ to say, with a slight abuse of notation, $\mu_i \in \mathcal{M}(\Omega)$ for every $i = 1, . . . , d$. A similar abuse of notation will be made when, for a vector field $v: \Omega \to \mathbf{R}^d$, we will write $v \in L^p(\Omega)$, to mean that $v_i \in L^p(\Omega)$ for every $i = 1, . . . , d$.

\medskip

Given $n_0 \in L^\infty(\mathbf{R}^d)$ such that $\nabla n_0 \in \mathcal{M}_{loc}(\mathbf{R}^d)$ and $\rho_0 \in BV(\mathbf{R}^d)$ such that $0 \le \rho_0 \le 1$ a.e., we denote by $(n_D, \rho_D, p_D)$ the unique weak solution to \eqref{eq:model_eqnts}  in the sense of Definition~\ref{eq:weak_solutions} with $D > 0$ and initial value $(n_0, \rho_0)$. Similarly $(n, \rho, p)$ denotes the weak solution with same initial value and $D = 0$. We also define, for $t \ge 0$
\begin{equation}
\Gamma_D(t) :=\partial \{p_D(t) > 0\},\quad\Gamma(t) := \partial \{ p(t) > 0\}.
\end{equation}
When the initial value $\rho_0$ is a characteristic function, for $t \ge 0$ and $D > 0$ we  will denote by $d_H(\Gamma_D(t), \Gamma(t))$ the Hausdorff distance between $\Gamma_D(t)$ and $\Gamma(t)$, which we recall is defined as
\begin{equation}\label{eq:haus_dist_def}
d_H(\Gamma_D(t), \Gamma(t)) := \max \left( \sup_{x \in \Gamma(t)} d(x, \Gamma_D(t)), \sup_{x \in \Gamma_D(t)} d(x, \Gamma(t)) \right).
\end{equation}

Roughly speaking, Theorem~\ref{thm:conv_patches} says that we have stability in the $D$ parameter provided the solution to the system \eqref{eq:model_eqnts} without nutrients diffusion is sufficiently regular. 

\begin{theorem}[Hausdorff convergence of the tumor patches]\label{thm:conv_patches}
Let $n_0:\R^d\to [\lambda, C]$ for some $0< \lambda \leq C<\infty$ such that $\nabla n_0 \in L^2_{loc}(\mathbf{R}^d)$. Consider the initial density  $\rho_ 0\in BV(\mathbf{\R}^d)$ which takes value of $0$ or $1$. For these choices of $(n_0, \rho_0)$, assume that either

\begin{enumerate}[(i)]
\item $n_0>1$ or\label{item:n_0_1}
\item there exist $\hat{t} \geq 0$ such that $\Gamma(t)$ is uniformly $C^1$ for every $t\ge\hat{t}$.\label{item:regularity_ass}
\end{enumerate}
Then for every $t>\hat{t} $ we have
\begin{equation}
\lim_{D\to 0} d_{H}\left( \Gamma_{D}(t), \Gamma(t)\right) = 0.
\end{equation}
\end{theorem}

The strict positivity of $n_0$ ensures that the tumor continues to expand over time. The regularity assumption \eqref{item:regularity_ass}, shown in many cases in \cite{Jacobs2022} (see the remark below), is to ensure that there is a stable direction of propagation for the pressure as $D$ tends to zero. These assumptions are likely not sharp, but we keep them in place to clarify our analysis.

\medskip

\begin{remark}
It is shown that a sufficient amount of nutrients leads to the regularity of $\Gamma(t)$, namely that \eqref{item:n_0_1} implies \eqref{item:regularity_ass} in the above theorem: see Theorem 2.7 of \cite{Jacobs2022}. When \eqref{item:n_0_1} does not hold \eqref{item:regularity_ass} still holds if the tumor manages to grow out of  the convex hull of its initial support $\Omega_0$:  See \cite[Lemma 4.6 and Theorem 2.7]{Jacobs2022} for the precise conditions that ensure that \eqref{item:regularity_ass} is satisfied.
\end{remark}
One of the main ingredients in the proof  is the following result about the strong $H^1$ convergence of the pressure variables, which is of independent interest. We will use the following notation: if $\rho$ is a measurable function such that $0 \le \rho \le 1$ almost everywhere, we define
\begin{equation}\label{eq:a_def_for_set}
H_{\rho} := \left\{ \xi \in H^1(\mathbf{R}^d):\ \xi \ge 0,\ \xi(1-\rho) = 0\ \text{a.e.}\right\}.
\end{equation}
\begin{proposition}[Strong convergence of $\nabla p_D$]\label{prop:strong_conv_pressure}
Let $n_0 \in L^{\infty}(\mathbf{R}^d)$ with $\nabla n_0 \in L^2_{loc}(\mathbf{R}^d)$. Let $\rho_0 \in L^1(\mathbf{R}^d)$ be compactly supported and such that $0 \le \rho_0 \le 1$. Then for Lebesgue point $t > 0$ of $p_D$ we have
\begin{equation}\label{eq:variational_ineq_pressure}
\int_{\mathbf{R}^d} \nabla p_D(x,t)\cdot \left( \nabla p_D(x,t) - \nabla \xi(x)\right)dx \le \int_{\mathbf{R}^d}(p_D(x, t) - \xi(x)) n_D(x, t)\rho_D(x, t)dx,
\end{equation}
for all $\xi \in H_{\rho_D(t)}$. In particular, for any $T > 0$, $ \nabla p_D \to \nabla p$ in $L^2((0, T) \times \mathbf{R}^d)$  as $D\to 0$.
\end{proposition}
\begin{remark}\label{rem:density_smooth}
Let us point out the following:
\begin{enumerate}
\item By Definition~\ref{eq:weak_solutions}  for every $D \ge 0$ the map $[0, T] \ni t \mapsto \rho_D(t) \in L^1(\mathbf{R}^d)$ is continuous, in particular  $\rho_D(t) \in L^1(\mathbf{R}^d)$ as well as the set $H_{\rho_D(t)}$ are well-defined for each $t>0$.\label{item:def_H_RHO}
\item By Proposition~\ref{prop:strong_conv_pressure} we know that for every $T>0$, $\nabla p_D \to \nabla p$ in $L^2((0, T), L^2(\mathbf{R}^d))$ as $D \to 0$. By the Sobolev embedding theorem this implies that $p_D\to p$ in $L^2((0, T), L^{2^*}(\mathbf{R}^d))$ as $D \to 0$. Using Lemma~\ref{lem:bound_supp} in Section~\ref{sec:proofs}, we infer that $p_D\to p$  in $L^q((0, T) \times \mathbf{R}^d)$  for every $1 \le q \le 2$.\label{item:pressure_conver}
\end{enumerate}
\end{remark}

\section{Preliminaries}\label{sec:prelim}
\subsection{Existence of weak solutions}
Here, we give the definition of weak solution to \eqref{eq:model_eqnts}. Our definition is a slightly modified version of \cite[Definition 2.1]{Jacobs2022} 
\begin{definition}\label{eq:weak_solutions}
Let $n_0 \in L^\infty(\mathbf{R}^d)$ such that $\nabla n_0 \in \mathcal{M}_{loc}(\mathbf{R}^d)$ and $\rho_0 \in BV(\mathbf{R}^d)$ such that $\rho_0 \in [0, 1]$. Fix $T>0$ and denote $Q_T := \mathbf{R}^d \times [0, T)$. Non-negative functions $\rho \in C([0, T], L^1(\mathbf{R}^d)) \cap L^{\infty}([0, T], BV(\mathbf{R}^d))$, $p \in L^2([0, T], H^1(\mathbf{R}^d))$, and $n \in L^{\infty}(Q_T)$ such that $\nabla n \in L^{\infty}([0, T], \mathcal{M}_{loc}(\mathbf{R}^d))$ are said to form a weak solution to \eqref{eq:model_eqnts} in $Q_T$, if they satisfy:
\begin{enumerate}[(i)]
\item $\rho \in [0, 1]$ and $p(1-\rho) = 0$ in $Q_T$.\label{item:orthog}
\item For any $\psi \in C^{\infty}_c(Q_T)$ we have
\begin{align}
\int_0^T \int_{\mathbf{R}^d} (\nabla \psi \cdot \nabla p - \rho \partial_t\psi) dx dt = \int_{\mathbf{R}^d} \psi(x, 0)\rho_0 dx + \int_0^T \int_{\mathbf{R}^d} \psi(n-b)\rho dx dt.
\end{align}\label{item:rhoeq}
\item For every $\psi \in C^{\infty}_c(Q_T)$ we have
\begin{align}
\int_0^T \int_{\mathbf{R}^d} (D\nabla n \cdot \nabla \psi - n\partial_t\psi) dx dt = \int_{\mathbf{R}^d}\psi(x, 0) n_0 dx - \int_0^T\int_{\mathbf{R}^d} \psi \rho n dx dt.
\end{align}\label{item:nutrienteq}
\end{enumerate}
\end{definition}
\begin{remark}
The difference between Definition~\ref{eq:weak_solutions} and \cite[Definition 2.1]{Jacobs2022} is that here we allow the initial nutrient $n_0$ to be such that $\nabla n_0 \in \mathcal{M}_{loc}(\mathbf{R}^d)$ instead of requiring $\nabla n_0 \in \mathcal{M}(\mathbf{R}^d)$. Correspondingly, we require $\nabla n \in L^{\infty}([0, T], \mathcal{M}_{loc}(\mathbf{R}^d))$ instead of requiring $\nabla n \in L^{\infty}([0, T], \mathcal{M}(\mathbf{R}^d))$. Also, the class of test functions for items \eqref{item:rhoeq} and \eqref{item:nutrienteq} in Definition~\ref{eq:weak_solutions} is now restricted to $C^{\infty}_c(Q_T)$. 
\end{remark}
The following well-posedness result is proved in \cite[Theorem 2.2]{Jacobs2022} for the original definition of weak solution \cite[Definition 2.1]{Jacobs2022}, but it can be proved in a completely analogous way with our definition.
\begin{theorem}[Theorem 2.2 in \cite{Jacobs2022}]\label{thm:ex_uniq}
Let $n_0 \in L^{\infty}(\mathbf{R}^d)$ with $\nabla n_0 \in \mathcal{M}_{loc}(\mathbf{R}^d)$. Let $\rho_0 \in BV(\mathbf{R}^d)$ be compactly supported and such that $0 \le \rho_0 \le 1$. Then for given $D, b \ge 0$ and any $T>0$ there exists a unique weak solution $(n, \rho, p)$ to \eqref{eq:model_eqnts} in $Q_T := \mathbf{R}^d \times [0, T)$ in the sense of Definition~\ref{eq:weak_solutions}.
\end{theorem}
\section{Construction of weak solutions}\label{sec:schemes}
 In \cite{Jacobs2022}, the authors provide a numerical scheme for constructing weak solutions. This is based on Wasserstein's projections, which we recall below. We also describe a slight modification of the scheme that ensures a stronger convergence for the nutrient variable: such convergence will be needed in the proof of Proposition~\ref{prop:strong_conv_pressure}. 
 
 Given a non-negative function $f \in L^1(\mathbf{R}^d)$ we denote by
\begin{equation}
M_f := \left\{ \rho \in L^1(\mathbf{R}^d):\ 0 \le \rho \le 1,\ \Vert \rho \Vert_{L^1(\mathbf{R}^d)} = \Vert f \Vert_{L^1(\mathbf{R}^d)}\right\}.
\end{equation}
\begin{enumerate}[(I)]
\item \textbf{Scheme I.} Let $n_0, \rho_0$ be as in Definition~\ref{eq:weak_solutions} and let $D\ge 0$. Fix a parameter $\tau > 0$, define $\rho_D^{0, \tau} = \rho_0,\ p_D^{0, \tau} = 0$ and $n_D^{0, \tau} = n_0$. We then define inductively for every $k \in \mathbf{N}$ functions $\rho^{k+1,\tau}_D, p^{k+1, \tau}_D$ and $n^{k+1, \tau}_D$ -- provided $\rho^{k,\tau}_D, p^{k, \tau}_D$ and $n^{k, \tau}_D$ are defined -- in the following way:
\begin{align}
& \rho_D^{k+1, \tau} =\underset{\rho \in M_{\rho^{k,\tau}_D}}{ \operatorname{argmin}} \frac{1}{2\tau} W_2^2(\rho, \rho_D^{k, \tau}( 1  + \tau n_D^{k,\tau})),
\\ & n^{k+1, \tau}_D = e^{\tau D\Delta}(n^{k, \tau}_D(1-\tau \rho^{k+1, \tau}_D))
\\ &p_D^{k+1, \tau} =\underset{p \ge 0}{ \operatorname{argmax}} \int_{\mathbf{R}^d} p^{c_\tau} \rho_D^{k, \tau}(1+\tau n_D^{k,\tau}) dx - \int_{\mathbf{R}^d} p dx,
\end{align}
where the $c_\tau$-transform of $p$ is defined as
\begin{equation}
p^{c_{\tau}}(x) = \inf_{y \in \mathbf{R}^d} p(y) + \frac{1}{2\tau}|x-y|^2.
\end{equation}
For future reference we also define
\begin{equation}
\mu_D^{k, \tau} := \rho_D^{k,\tau}(1+\tau n_D^{k, \tau}).
\end{equation}
We then define $\rho_D^\tau, p_D^\tau, n^{\tau}_D$ and $\mu_D^{\tau}$ as the piecewise constant in time, right-continuous interpolations of the above-defined functions, for example
\begin{align*}
\rho_D^{\tau} (x, t): = \rho_D^{ k+1, \tau}(x)\quad x \in \mathbf{R}^d,\ t \in [k\tau, (k+1)\tau).
\end{align*}
When $D = 0$, we will sometimes drop the subscript $D$ in the above notation.\label{item:first_scheme}
\item \textbf{Scheme II.} Let $n_0, \rho_0$ be as in Definition~\ref{eq:weak_solutions} and let $D\ge 0$. Fix a parameter $\tau > 0$, define $\rho_D^{0, \tau} = \rho_0,\ p_D^{0, \tau} = 0$ and define $n^\tau_D|_{[0, \tau)}$ as the solution to
\begin{align}
\begin{cases}
\partial_t n^{\tau}_D - D\Delta n^\tau_D = -n^\tau_D \rho_0\ &\text{on}\ \mathbf{R}^d \times (0, \tau),
\\ n^\tau_D(0) = n_0\ &\text{on}\ \mathbf{R}^d.
\end{cases}
\end{align}
We then define inductively for every $k \in \mathbf{N}$ functions $\rho^{k+1,\tau}_D, p^{k+1, \tau}_D$ and $n^{\tau}_D|_{[(k+1)\tau, (k+2)\tau)}$ -- provided $\rho^{k,\tau}_D, p^{k, \tau}_D$ and $n^{\tau}_D|_{[k\tau, (k+1)\tau)}$ are defined -- in the following way:
\begin{align}
\rho_D^{k+1, \tau} = \operatorname{argmin}_{\rho \in M_{\rho^{k,\tau}_D}} \frac{1}{2\tau} W_2^2(\rho, \rho_D^{k, \tau}( 1  + \tau n_D^{\tau}(k\tau))),
\end{align}
the function $n^\tau_D|_{[(k+1)\tau, (k+2)\tau)}$ is defined as the solution of
\begin{align}
\begin{cases}
\partial_t n^{\tau}_D - D\Delta n^\tau_D = -n^\tau_D \rho^{k, \tau}\ &\text{on}\ \mathbf{R}^d \times ((k+1)\tau, (k+2)\tau),
\\n^\tau_D((k+1)\tau) = \lim_{s\uparrow (k+1)\tau} n^\tau_D(s)\ &\text{on}\ \mathbf{R}^d.
\end{cases},
\end{align}
while
\begin{align}
p_D^{k+1, \tau} = \underset{{p \ge 0}}{\operatorname{argmax}} \int_{\mathbf{R}^d} p^{c_\tau} \rho_D^{k, \tau}(1+\tau n_D^{\tau}(k\tau)) dx - \int_{\mathbf{R}^d} p dx,
\end{align}
where the $c_\tau$-transform of $p$ is defined as
\begin{equation}
p^{c_{\tau}}(x) = \inf_{y \in \mathbf{R}^d} p(y) + \frac{1}{2\tau}|x-y|^2.
\end{equation}
For future reference we also define
\begin{equation}
\mu_D^{k, \tau} := \rho_D^{k,\tau}(1+\tau n_D^{\tau}(k\tau)).
\end{equation}
We then define $\rho_D^\tau, p_D^\tau$ and $\mu_D^{\tau}$ as the piecewise constant in time, right-continuous interpolations of the above-defined functions, for example
\begin{align*}
\rho_D^{\tau} (x, t) = \rho_D^{k+1, \tau}(x)\quad x \in \mathbf{R}^d,\ t \in [k\tau, (k+1)\tau).
\end{align*}
When $D = 0$, we will sometimes drop the subscript $D$ in the above notation.\label{item:second_scheme}
\end{enumerate}
After choosing one of the schemes above, one obtains a family $\{(n^\tau_D, \rho^{\tau}_D, p^{\tau}_D)\}_{\tau > 0}$ of approximate solutions to \eqref{eq:model_eqnts}. The convergence of the first scheme is proved in \cite[Proposition~3.6]{Jacobs2022}. We will prove the convergence of the second scheme in Theorem~\ref{thm:estimate_n_second_conv}. The advantage of the first scheme is that the needed estimates in the proof of its convergence are independent of $D \ge 0$, in particular, this will allow us to prove the convergence of $(n_D, \rho_D, p_D)$ to $(n, \rho, p)$ as $D \to 0$ (cf.\ Theorem~\ref{thm:conv_D_zero}). The advantage of the second scheme is that \emph{for a fixed $D \ge 0$} one has better convergence properties of $n^\tau_D$ to $n_D$ as $\tau \to 0$, which we will need in Proposition~\ref{prop:strong_conv_pressure}.
\begin{theorem}\label{thm:conv_D_zero}
Let $n_0 \in L^{\infty}(\mathbf{R}^d)$ with $\nabla n_0 \in \mathcal{M}_{loc}(\mathbf{R}^d)$. Let $\rho_0 \in BV(\mathbf{R}^d)$ be such that $0\leq \rho_0\leq 1$ and  $\operatorname{supp}(\rho_0) \subset B_R$. Let $(n_D, \rho_D, p_D)$ be the unique weak solution to \eqref{eq:model_eqnts} in the sense of Definition~\ref{eq:weak_solutions} for $D>0$, and denote by $(n, \rho, p)$ the unique weak solution for $D=0$.
 Then the following holds for $ t\in [0,T]$: 
\begin{enumerate}[(i)]
\item $\Vert \rho_D(\cdot, t) \Vert_{L^1(\mathbf{R}^d)} \le B(t):= e^{t\Vert n_0 \Vert_{L^{\infty}}}\Vert \rho_0 \Vert_{L^1(\mathbf{R}^d)}$.\label{item:est_rho}
\item $\Vert p_D \Vert_{L^2(\mathbf{R}^d \times [0, t])} \le CB(t)^{\frac{d+4}{2d}}\Vert n_0 \Vert_{L^{\infty}(\mathbf{R}^d)}^{1/2}$.\label{item:est_p}
\item $\Vert \nabla p_D\Vert_{L^2(\mathbf{R}^d\times [0, t])}\le CB(t)^{\frac{d+2}{2d}}\Vert n_0 \Vert_{L^{\infty}(\mathbf{R}^d)}^{1/2}$.\label{item:est_nabla_p}
\item For $R(t)$ given in \eqref{eq:def_rad} of Lemma \ref{lem:bound_supp}, 
\begin{align*}
\Vert \nabla \rho_D(\cdot, t) \Vert_{L^1(\mathbf{R}^d)} + \Vert \nabla n_D(\cdot, t) \Vert_{\mathcal{M}(B_{R(T)})} \le e^{(2\Vert n_0 \Vert_{L^{\infty}(\mathbf{R}^d)} + 1)t}\bigg( &\Vert \nabla \rho_0 \Vert_{\mathcal{M}(\mathbf{R}^d)} 
\\ &+ \Vert \nabla n_0 \Vert_{\mathcal{M}(B_{R(T)})}\bigg).
\end{align*}\label{item:bound_grad_nut}
\end{enumerate}
Moreover, for every $t, s > 0$ such that $t+s \le T$,
\begin{equation}\label{eq:l1_equicont}
\Vert \rho_D(\cdot, t+s) - \rho_D(\cdot, t)\Vert_{L^1(\mathbf{R}^d)} \le s\Vert n_0 \Vert_{L^{\infty}(\mathbf{R}^d)}\Vert \rho_D \Vert_{L^{\infty}([0,T], L^1(\mathbf{R}^d))}.
\end{equation}
In particular, the following holds as $D\to 0$:
\begin{enumerate}[(a)]
\item For every $t \in [0, T]$, $\rho_D(\cdot, t)\to \rho(\cdot,t)$ in $L^1(\mathbf{R}^d)$.\label{item:rho_conv}
\item $p_D \xrightharpoonup{} p$ and $\nabla p_D \xrightharpoonup{} \nabla p$ weakly in $L^2((0, T), L^2(\mathbf{R}^d))$.\label{item:weak_conv_grad}
\item $n_D$ converges weakly-$*$ to $n$ in $L^{\infty}((0, T) \times \mathbf{R}^d)$.\label{item:conv_nut}
\end{enumerate}
\end{theorem}
\begin{theorem}\label{thm:estimate_n_second_conv}
Fix $D \ge 0$. Let $n_0 \in L^{\infty}(\mathbf{R}^d)$ with $\nabla n_0 \in L^2_{loc}(\mathbf{R}^d)$. Let $\rho_0 \in BV(\mathbf{R}^d)$ be compactly supported, such that $\rho_0 \in [0, 1]$. Let $\{(n^\tau_D, \rho^\tau_D, p^\tau_D)\}_{\tau > 0}$ be the family of approximate solutions generated by Scheme \ref{item:second_scheme}. Then the following holds as $\tau \to 0$: 
\begin{enumerate}[(i)]
\item $\rho_D^{\tau}\to \rho_D$ in $L^1((0, T)\times \mathbf{R}^d)$.\label{item:rho_conv_second}
\item $p_D^{\tau}\xrightharpoonup{} p_D$ weakly in $L^2((0, T)\times \mathbf{R}^d)$.\label{item:p_conv_second}
\item $\nabla p_D^{\tau}\xrightharpoonup{}\nabla p_D$ weakly in $L^2((0, T)\times \mathbf{R}^d)$. \label{item:nabla_p_conv_second}
\item $n_D^{\tau}\to n_D$ in $L_{loc}^2((0, T)\times \mathbf{R}^d)$.\label{item:n_conv_second}
\end{enumerate}
\end{theorem}

\section{Proofs}\label{sec:proofs}

\subsection{Convergence of the numerical schemes}
Before proving Theorem~\ref{thm:conv_D_zero} and Theorem~\ref{thm:estimate_n_second_conv}, we state the following comparison principle, which will allow us to control the size of the support of $\rho_D^\tau$ when the initial value is compactly supported.
\begin{proposition}\label{prop:comp}
Let $f \in L^\infty(\mathbf{R}^d), f \ge 0$, assume that $\tilde{\rho}_{0,1}, \tilde{\rho}_{0,2} \in L^{1}(\mathbf{R}^d)$ are such that $0 \le \tilde{\rho}_{0, 1} \le \tilde{\rho_{0, 2}} \le 1$. For a given $\tau > 0$ and any $\lambda \ge \Vert f \Vert_{L^{\infty}(\mathbf{R}^d)}$ define
\begin{align}
\tilde{\rho}_{1} &:= \underset{ \tilde{\rho} \in M_{\tilde{\rho}_{0,1}f}}{\operatorname{argmin}} \frac{1}{2\tau}W_2^2(\tilde{\rho}, \tilde{\rho}_{0,1}f)
\\ \tilde{\rho}_2 &:= \underset{\tilde{\rho} \in M_{\tilde{\rho}_{0,2}\lambda}}{\operatorname{argmin}} \frac{1}{2\tau}W_2^2(\tilde{\rho}, \tilde{\rho}_{0,2}\lambda).
\end{align}
Then it holds
\begin{equation}
\tilde{\rho}_1 \le \tilde{\rho}_2\quad \text{almost everywhere on}\ \mathbf{R}^d.
\end{equation}
\end{proposition}
\begin{lemma}\label{lem:bound_supp}
Let $\rho_0 \in L^1(\mathbf{R}^d)$ such that $0 \le \rho_0 \le 1$ and such that $\operatorname{supp}(\rho_0) \subset B_R$, then, regardless of the choice of the approximation scheme, we have that for every $D \ge 0$ and every $\tau > 0$
\begin{equation}\label{eq:claim_supports}
\operatorname{supp}(\rho_D^{\tau}(t)) \subset B_{R(t)},
\end{equation}
where
\begin{equation}\label{eq:def_rad}
R(t) = R\left( e^{t\Vert n_0 \Vert_{L^{\infty}(\mathbf{R}^d)}}\right)^{\frac{1}{d}}.
\end{equation}
\end{lemma}
\begin{proof}[Proof of Proposition~\ref{prop:comp}]
By an application of \cite[Theorem 5.1]{Alexander2014} with $m = 0, \Phi = 0$ and
\begin{equation}
\rho_{0,1} = \frac{\tilde{\rho}_{0,1}f}{\lambda},\quad \rho_{0,2} = \tilde{\rho}_{0,2},
\end{equation}
we get that
\begin{equation}
\rho_1 \le \rho_2,\quad \text{almost everywhere on}\ \mathbf{R}^d,
\end{equation}
where
\begin{align*}
\rho_{1} &:= \operatorname{argmin}_{ \rho \in M_{\rho_{0,1}}} \frac{1}{2\tau}W_2^2(\rho, \rho_{0,1})
\\ \rho_2 &:= \operatorname{argmin}_{ \rho \in M_{\rho_{0,2}}} \frac{1}{2\tau}W_2^2(\rho, \rho_{0,2}).
\end{align*}
The proof is then completed by observing that 
\begin{equation}
\tilde{\rho}_1 = \rho_1 \lambda,\quad \tilde{\rho}_2 = \rho_2 \lambda.\qedhere
\end{equation}
\end{proof}
\begin{proof}[Proof of Lemma~\ref{lem:bound_supp}]
First, we show that for any $1 \le k \le \left[ \frac{t}{\tau} \right]$ we have
\begin{equation}\label{eq:step_claim}
\operatorname{supp}(\rho_D^{k, \tau}) \subset B_{R\left(1 + \tau \Vert n_0 \Vert_{L^{\infty}(\mathbf{R}^d)}\right)^{\frac{k}{d}}}.
\end{equation}
The proof of \eqref{eq:claim_supports} is then concluded by observing that
\begin{equation}
(1+\tau\Vert n_0 \Vert_{L^{\infty}(\mathbf{R}^d)})^k \le e^{k\tau \Vert n_0 \Vert_{L^{\infty}(\mathbf{R}^d)}}.
\end{equation}
To show \eqref{eq:step_claim} we proceed by induction. We assume that 
\begin{align}\label{eq:inductive_ass}
\operatorname{supp}(\rho_D^{k-1, \tau}) \subset B_{R_{k-1}},\quad R_{k-1} := R\left(1 + \tau \Vert n_0 \Vert_{L^{\infty}(\mathbf{R}^d)}\right)^{\frac{{k-1}}{d}}
\end{align}
We let $\lambda = 1 + \tau \Vert n_0 \Vert_{L^{\infty}(\mathbf{R}^d)}$ and we let $\tilde{\rho}_1$ be defined as
\begin{equation}
\tilde{\rho}_1 := \underset{\tilde{\rho} \in M_{\lambda\mathbf{1}_{B_{R_{k-1}}}}}{\operatorname{argmin}} \frac{1}{2\tau} W_2^2(\tilde{\rho}, \mathbf{1}_{B_{R_{k-1}}}\lambda).
\end{equation}
Then an application of Proposition~\ref{prop:comp} yields
\begin{equation}
\rho_D^{k, \tau} \le \tilde{\rho}_1.
\end{equation}
By symmetry considerations we have $\tilde{\rho}_1 = B_{\lambda^{1/d}R_{k-1}} = B_{R_{k}}$. This yields
\begin{equation}
\operatorname{supp}(\rho^{k, \tau}_D) \subset B_{R_k}.
\end{equation}
Since \eqref{eq:inductive_ass} holds true for $k=1$ by assumption, this concludes the proof of \eqref{eq:step_claim}.
\end{proof}
\begin{proof}[Proof of Theorem~\ref{thm:conv_D_zero}] 
We prove Theorem~\ref{thm:conv_D_zero} in three steps.

Step 1. For each fixed $D \ge 0$ estimates \eqref{item:est_rho}, \eqref{item:est_p}, \eqref{item:est_nabla_p} and \eqref{item:bound_grad_nut} in Theorem~\ref{thm:conv_D_zero} follow from the corresponding estimates \cite[Lemma 3.4]{Jacobs2022} for the approximations $\{(n^\tau_D, \rho^\tau_D, p^\tau_D)\}_{\tau > 0}$ obtained by using Scheme \ref{item:first_scheme}. 
\\
To see this, observe that the estimates for the approximations are proved in \cite[Lemma 3.4]{Jacobs2022} and the constants in the bounds are independent of $D$. Moreover, all those estimates pass to the limit as $\tau \to 0$ because the following items hold true as $\tau \to 0$ (see \cite[Proposition 3.6]{Jacobs2022}):
\begin{align}
    \rho_D^\tau &\to \rho_D\quad &&\text{strongly in}\ L^1((0, T) \times \mathbf{R}^d),
    \\ p_D^\tau &\xrightharpoonup{} p_D\quad &&\text{weakly in}\ L^2((0, T) \times \mathbf{R}^d),
    \\ \nabla p_D^\tau &\xrightharpoonup{} \nabla p_D\quad &&\text{weakly in}\ L^2((0, T) \times \mathbf{R}^d),
    \\ n_D^\tau &\xrightharpoonup{*} n_D\quad &&\text{weakly-$*$ in}\ L^\infty((0, T) \times \mathbf{R}^d).
\end{align}
  Estimate \eqref{eq:l1_equicont} in Theorem~\ref{thm:conv_D_zero} follows from the analogous estimate at the level of the approximation \cite[Lemma 3.5]{Jacobs2022}, and the fact that $\rho_D^\tau(t)$ converges for almost every $t \in [0, T]$ strongly in $L^1(\mathbf{R}^d)$, this yields \eqref{eq:l1_equicont} in Theorem~\ref{thm:conv_D_zero} almost everywhere. To upgrade to every time $t \in [0, T]$ one has just to recall that $\rho_D \in C([0, T], L^1(\mathbf{R}^d))$. 

Step 2. We show that for any sequence $D_j$ converging to zero as $j \to +\infty$, there exists a triple $(\tilde{n}, \tilde{\rho}, \tilde{p})$ such that over a non-relabeled subsequence, as $j\to +\infty$,
\begin{align}
    \rho_{D_j} &\to \tilde{\rho}\quad &&\text{strongly in}\ L^1((0, T)\times \mathbf{R}^d),
    \\ n_{D_j} &\xrightharpoonup{*} \tilde{n}\quad &&\text{weakly-$*$ in}\ L^\infty((0, T)\times \mathbf{R}^d),
    \\ \nabla n_{D_j} &\xrightharpoonup{*} \nabla \tilde{n}\quad &&\text{weakly-$*$ in}\ \mathcal{M}_{loc}((0, T)\times \mathbf{R}^d),
    \\ p_{D_j} &\xrightharpoonup{} \tilde{p}\quad &&\text{weakly in}\ L^2((0, T)\times \mathbf{R}^d),
    \\ \nabla p_{D_j} &\xrightharpoonup{} \nabla \tilde{p}\quad &&\text{weakly-$*$ in}\ L^2((0, T)\times \mathbf{R}^d).
\end{align}
To do so, pick a sequence $D_j$ such that $D_j \to 0$ as $j$ converges to infinity. Using estimate \eqref{eq:l1_equicont} in Theorem~\ref{thm:conv_D_zero} and items \eqref{item:est_rho} and \eqref{item:bound_grad_nut} in Theorem~\ref{thm:conv_D_zero} we can apply Kolmogorov-Riesz-Fischer's theorem to infer that the sequence $\rho_{D_j}$ is precompact in $L^1((0, T) \times \mathbf{R}^d)$, so that we can assume without loss of generality $\lim_{j \to +\infty} \rho_{D_j} = \tilde{\rho}$ in $L^1((0, T) \times \mathbf{R}^d)$ for some $\tilde{\rho} \in L^1((0, T) \times \mathbf{R}^d)$ -- observe also that because of \eqref{eq:l1_equicont} in Theorem~\ref{thm:conv_D_zero} we get that $\rho_{D_j}(\cdot, t)$ converges to $\tilde{\rho}(\cdot, t)$ in $L^1(\mathbf{R}^d)$ for every $t \in [0, T]$. From the bound on the $L^{\infty}$ norm of $n_{D_j}$ we get that, up to extracting a subsequence, $n_{D_j}$ converges weakly-$*$ in $ L^{\infty}((0, T) \times \mathbf{R}^d)$ to a function $\tilde{n} \in L^{\infty}((0, T) \times \mathbf{R}^d)$. From the bound on the gradient of $n_{D_j}$ (Item \eqref{item:bound_grad_nut} in Theorem~\ref{thm:conv_D_zero}) we also obtain that $\nabla n_{D_j}$ converges weakly-$*$ in the sense of Radon measures to $\nabla n$. From items \eqref{item:est_p} and \eqref{item:est_nabla_p} in Theorem~\ref{thm:conv_D_zero} we infer that, up to a subsequence, there exists $\tilde{p} \in H^1((0, T) \times \mathbf{R}^d)$ such that $p_{D_j}$ converges weakly to $p$ in $L^2((0, T) \times \mathbf{R}^d)$ and $\nabla p_{D_j}$ converges weakly to $\nabla p$ in $L^2((0, T) \times \mathbf{R}^d)$. 

Step 3. We claim that for any limit point $(\tilde{n}, \tilde{\rho}, \tilde{p})$ as in Step 2 we have
\begin{equation}\label{eq:unique_lp}
(\tilde{n}, \tilde{\rho}, \tilde{p}) = (n, \rho, p).
\end{equation}
This says that the limit point is unique, and thus proves items \eqref{item:rho_conv}, \eqref{item:weak_conv_grad}, \eqref{item:conv_nut} in Theorem~\ref{thm:conv_D_zero}.
\\
To verify \eqref{eq:unique_lp} observe that thanks to the uniqueness result of Theorem~\ref{thm:ex_uniq}, we just need to show that $(\tilde{n}, \tilde{\rho}, \tilde{p})$ is a weak solution to \eqref{eq:model_eqnts} in the sense of Definition~\ref{eq:weak_solutions} with $D = 0$ and initial value $(n_0, \rho_0)$. We have to check items \eqref{item:orthog}, \eqref{item:rhoeq} and \eqref{item:nutrienteq} in Definition~\ref{eq:weak_solutions}. For Item \eqref{item:orthog} in Definition~\ref{eq:weak_solutions} we pick a function $\varphi \in C_c((0, T) \times \mathbf{R}^d)$ and, using the fact that $p_{D_j}(1-\rho_{D_j}) = 0$, we obtain
\begin{align*}
\int_{0}^T \int_{\mathbf{R}^d} \tilde{p}(1-\tilde{\rho})\varphi dx dt = \lim_{j \to +\infty} \int_0^T \int_{\mathbf{R}^d} p_{D_j}(1 - \rho_{D_j}) \varphi dx dt = 0,
\end{align*}
and by the arbitrary choice of $\varphi$ we conclude that $\tilde{p}(1-\tilde{\rho}) = 0$ almost everywhere on $Q_T$. For Item \eqref{item:rhoeq} in Definition~\ref{eq:weak_solutions} we pick $\psi \in C^{\infty}_c(Q_T)$ and, using Item \eqref{item:rhoeq} in Definition~\ref{eq:weak_solutions} for $(n_{D_j}, \rho_{D_j}, p_{D_j})$, we obtain
\begin{align*}
\int_0^T \int_{\mathbf{R}^d} (\nabla \psi \cdot \nabla \tilde{p} - \tilde{\rho} \partial_t \psi) dx dt &= \lim_{j \to +\infty} \int_0^T \int_{\mathbf{R}^d} (\nabla \psi \cdot \nabla p_{D_j} - \rho_{D_j} \partial_t \psi) dx dt
\\ &= \lim_{j \to +\infty} \int_{\mathbf{R}^d} \psi(x, 0)\rho_0 dx + \int_0^T \int_{\mathbf{R}^d}\psi n_{D_j} \rho_{D_j} dx dt
\\ &= \int_{\mathbf{R}^d} \psi(x, 0)\rho_0 dx + \int_0^T \int_{\mathbf{R}^d}\psi \tilde{n} \tilde{\rho} dx dt,
\end{align*}
which gives Item \eqref{item:rhoeq} in Definition~\ref{eq:weak_solutions} for $(\tilde{n}, \tilde{\rho}, \tilde{p})$. The last item in Definition~\ref{eq:weak_solutions} is checked analogously. Thus $(\tilde{n}, \tilde{\rho}, \tilde{p})$ is a weak solution of \eqref{eq:model_eqnts} in the sense of Definition \ref{eq:weak_solutions} and the proof is complete.
\end{proof}
\begin{proof}[Proof of Theorem~\ref{thm:estimate_n_second_conv}]
We divide the proof into five steps.

Step 1. Spacial gradient bound on $\rho_D^\tau$. We claim that for every $D \ge 0$ and every $t \in (0, T)$ we have
\begin{align}\label{eq:step1}
\Vert \nabla \rho^{\tau}_D(\cdot, t) \Vert_{L^1(\mathbf{R}^d)} \le &\Vert \nabla \rho_0 \Vert_{L^1(\mathbf{R}^d)} + \Vert n_0 \Vert_{L^{\infty}} \Vert \nabla \rho_D^\tau \Vert_{L^1((0, t), \mathcal{M}(B_{R(t)}))}
\\ &+\sum_{k \le \left[\frac{t}{\tau}\right]}\tau \Vert \nabla n_D^\tau(k\tau) \Vert_{\mathcal{M}(B_{R(t)})},
\end{align}
where $R(t)$ is as in \eqref{eq:claim_supports} in Lemma~\ref{lem:bound_supp}.
\\
To show this, let $k \in \mathbf{N}$ such that $k \le \left[\frac{t}{\tau}\right] - 1$ and observe that by definition of $\rho_D^{k+1, \tau}$, by \cite[Theorem 1.1]{DePhilippis2016} and by \eqref{eq:claim_supports} in Lemma~\ref{lem:bound_supp} we have
\begin{align*}
\Vert \nabla \rho_D^{k+1, \tau} \Vert_{\mathcal{M}(\mathbf{R}^d)} &\le \Vert \nabla (\rho_D^{k, \tau}(1+\tau n_D^\tau(k\tau) )) \Vert_{\mathcal{M}(\mathbf{R}^d)}
\\ &\le (1+\tau \Vert n_0 \Vert_{L^{\infty}(\mathbf{R}^d)}) \Vert \nabla \rho^{k, \tau}_D \Vert_{\mathcal{M}(\mathbf{R}^d)} + \tau \Vert \nabla n_D^\tau(k\tau) \Vert_{	\mathcal{M}(B_{R(t)})}.
\end{align*}
Iterating the last inequality clearly yields \eqref{eq:step1}.

Step 2. We prove that for every $1 \gg D > 0$ the family $\{n_D^{\tau}\}_{\tau > 0}$ is bounded in $H_{loc}^1((0, T) \times \mathbf{R}^d)$. More precisely, we show that for every $R > 0$ and for every $1 \gg D > 0$ there exists a constant $C = C(D, R, T, n_0)$ such that for every $\tau > 0$
\begin{equation}\label{eq:bounds_H}
\max_{s \in [0, T]} \Vert n_D^\tau(s)\Vert_{L^2(B_R)}^2 + \max_{s \in [0, T]} \Vert \nabla n_D^\tau(s)\Vert_{L^2(B_R)}^2 + \Vert \partial_s n_D^\tau\Vert_{L^2((0, T) \times B_R)}^2 \le C.
\end{equation}
\\
This is proved in \cite[Lemma 3.5]{Maury2014} when $\nabla n_0 \in L^2(\mathbf{R}^d)$. We include a proof for the sake of completeness. Fix $R>0$ and a cut-off function $\eta \in C^{\infty}_c(\mathbf{R}^d)$ such that $\mathbf{1}_{B_{2R}} \le \eta \le \mathbf{1}_{B_{4R}}$. We first multiply the equation for $n^\tau_D$ on each interval $[k\tau, (k+1)\tau)$ by $n^\tau_D\eta^2$ and we integrate in space to get that on $(0, T)$
\begin{align*}
\frac{d}{dt}\frac{1}{2}\Vert n^\tau_D\eta \Vert_{L^2(\mathbf{R}^d)}^2 + D\int_{\mathbf{R}^d} |\nabla n_D^\tau |^2 \eta^2 dx + 2D\int_{\mathbf{R}^d} n^\tau_D\eta \nabla n^\tau_D \cdot \nabla \eta dx \le 0.
\end{align*}
Using the Cauchy-Schwarz inequality we obtain
\begin{align*}
\frac{d}{dt}\frac{1}{2}\Vert n^\tau_D\eta \Vert_{L^2(\mathbf{R}^d)}^2 + \frac{D}{2}\int_{\mathbf{R}^d} |\nabla n_D^\tau|^2 \eta^2 dx\le CD\Vert n_0 \Vert_{L^\infty(\mathbf{R}^d)}^2.
\end{align*}
Integrating the inequality in time we obtain
\begin{align*}
\sup_{t \in [0, T]} \Vert n_D^\tau(t)\eta \Vert_{L^2(\mathbf{R}^d)}^2 + \frac{D}{2}\int_0^T\int_{\mathbf{R}^d} |\nabla n_D^\tau|^2 \eta^2 dxdt\le C\Vert n_0 \Vert_{L^\infty(\mathbf{R}^d)}^2 + \Vert n_0 \eta \Vert_{L^2(\mathbf{R}^d)}^2.
\end{align*}
Thus we have that
\begin{align}\label{eq:first_bound}
\sup_{t \in [0, T]} \Vert n_D^\tau(t)\Vert_{L^2(B_{2R})}^2  + \Vert \nabla n_D^\tau \Vert_{L^2((0, T) \times B_R)}^2 \le C.
\end{align}
We  now take another cut-off function $\eta \in C^{\infty}_c(\mathbf{R}^d)$ such that $\mathbf{1}_{B_R} \le \eta \le \mathbf{1}_{B_{2R}}$, we multiply the equation for $n^\tau_D$ by $\partial_t n^\tau_D\eta^2$ and we integrate in space to get, after using Young's inequality,
\begin{align*}
&\frac{1}{4}\int_{\mathbf{R}^d}|\partial_tn^\tau_D|^2\eta^2 dx + D\frac{d}{dt}\frac{1}{2}\int_{\mathbf{R}^d} |\nabla n^\tau_D|^2 \eta^2 dx
\\ &\le C\Vert n_D^\tau \Vert_{L^2(B_{2R})} + C \int_{B_{2R}} |\nabla n_D^\tau|^2 dx.
\end{align*}
By first applying Gronwall's inequality and then integrating the previous inequality in time, we obtain, using also \eqref{eq:first_bound} 
\begin{align*}
&\int_0^T \int_{\mathbf{R}^d} |\partial_t n_D^\tau |^2 \eta^2 dx + \max_{s \in [0, T]}\int_{\mathbf{R}^d}|\nabla n_D^\tau(s)|^2 dx
\\ &\le C\left(\Vert n_0 \Vert_{L^{\infty}(\mathbf{R}^d)}^2 + \Vert n_0 \Vert_{L^2(B_{4R})}^2 + \Vert \nabla n_0 \Vert_{L^2(B_{2R})} \right).
\end{align*}
Recalling that $\mathbf{1}_{B_R} \le \eta$, the previous inequality together with \eqref{eq:first_bound} and the fact that $\nabla n_0 \in L^2_{loc}(\mathbf{R}^d)$ yields \eqref{eq:bounds_H}.

Step 3. We claim that there exists a constant $C = C(D, T, n_0, \rho_0, R)$ such that for every $D \ge 0$ we have
\begin{align}\label{eq:bound_grad_rho}
&\sup_{0 < \tau \ll 1} \sup_{t \in [0, T]}\Vert \nabla \rho^{\tau}_D(\cdot, t) \Vert_{L^1(\mathbf{R}^d)} \le C.
\end{align}
To see this we denote by $B$ the ball $B_{R(T)}$ appearing in Lemma~\ref{lem:bound_supp}. We distinguish two cases. If $D > 0$ we observe that by H\"{o}lder's inequality and by \eqref{eq:bounds_H} we have that there exists a constant $C$ such that
\begin{equation}
\Vert \nabla n_D^\tau(k\tau) \Vert_{\mathcal{M}(B)} \le |B|^{1/2}C.
\end{equation}
Thus \eqref{eq:step1} reads as
\begin{align*}
\Vert \nabla \rho^{\tau}_D(\cdot, t) \Vert_{L^1(\mathbf{R}^d)} \le &\Vert \nabla \rho_0 \Vert_{L^1(\mathbf{R}^d)} + \Vert n_0 \Vert_{L^{\infty}} \Vert \nabla \rho_D^\tau \Vert_{L^1((0, t), \mathcal{M}(B_{R(t)}))} 
\\ &+ |B|^{1/2}C,
\end{align*}
and an application of Gronwall's inequality yields \eqref{eq:bound_grad_rho}. 
\\
If $D = 0$ the explicit expression for the approximate nutrient variable is
\begin{equation}\label{eq:nutrient_exp_D_zero}
n^\tau(x, t) = n_0e^{-\int_0^t \rho^\tau(s-\tau)ds},\quad x \in \mathbf{R}^d, t \in (0, T), 
\end{equation}
and from this we also get an expression for its gradient
\begin{equation}\label{eq:nutrient_exp_D_zero_grad}
\nabla n^{\tau}(x, t) = e^{-\int_0^t \rho^{\tau}(x, s-\tau) ds} \left( \nabla n_0 (x) - \int_0^t \nabla \rho^{\tau}(x, s-\tau) ds\right).
\end{equation}
In particular, we can compute
\begin{align*}
|\nabla n^\tau(x, t)| &\le |\nabla n_0(x)| + \Vert n_0 \Vert_{L^{\infty}(\mathbf{R}^d)}|\nabla \rho^\tau(x, t-\tau)|,
\end{align*}
which, together with \eqref{eq:first_bound}, gives
\begin{align}
\Vert \nabla n^\tau(\cdot, t)\Vert_{L^1(B)} \le &\Vert \nabla n_0(x) \Vert_{L^1(B)} + \Vert n_0 \Vert_{L^{\infty}(\mathbf{R}^d)}\Vert \nabla \rho_0 \Vert_{L^1(\mathbf{R}^d)}
\\ &+ \Vert n_0 \Vert_{L^{\infty}(\mathbf{R}^d)}^2 \Vert \nabla \rho^\tau \Vert_{L^1([0, t)\times \mathbf{R}^d)} \label{eq:pre_gronwall_nut}
\\ &+  \Vert n_0 \Vert_{L^{\infty}(\mathbf{R}^d)} \Vert \nabla n^\tau \Vert_{L^1([0, t) \times B)}.
\\ &+\Vert n_0 \Vert_{L^{\infty}(\mathbf{R}^d)} \left( \sum_{k \le \left[ \frac{t}{\tau} \right]} \tau \Vert \nabla n^{\tau}(k\tau) \Vert_{\mathcal{M}(B)} -  \Vert \nabla n^\tau \Vert_{L^1([0, t) \times B)}\right)
\nonumber
\end{align}
Observe that the last term in \eqref{eq:pre_gronwall_nut} is $o(1)$ as $\tau \to 0$, this follows from the continuity of the map $t \mapsto \Vert \nabla n^{\tau}(x, t) \Vert_{\mathcal{M}(B)}$, which follows from \eqref{eq:nutrient_exp_D_zero_grad}. Summing \eqref{eq:step1} and \eqref{eq:pre_gronwall_nut} and using Gronwall's inequality we obtain the claim for $D = 0$.

Step 4. Precompactness of $\{n_D^\tau\}_{\tau > 0}$ in $L_{loc}^2((0, T) \times \mathbf{R}^d)$.
\\
For $D>0$, the precompactness of $\{n_D^\tau\}_{\tau > 0}$ in $L_{loc}^2((0, T) \times \mathbf{R}^d)$ is a consequence of Step 3 and the compact Sobolev embedding. For $D = 0$ we use the explicit expression for the nutrient variable \eqref{eq:nutrient_exp_D_zero} and the bound \eqref{eq:bound_grad_rho} to infer that 
\begin{equation}
\sup_{0 < \tau \ll 1} \Vert \nabla n^\tau \Vert_{L^1((0, T) \times B)} < \infty. 
\end{equation}
It is also clear from the equation
\begin{equation}
\partial_t n^\tau = -n^\tau \rho^{k-1},\quad t \in (k\tau, (k+1)\tau), k \in \mathbf{N},
\end{equation}
that
\begin{equation}
\sup_{0 < \tau \ll 1} \Vert \partial_t n^\tau \Vert_{L^1((0, T) \times B)} < \infty. 
\end{equation}
In particular, $\{n^\tau\}$ is bounded in $BV_{loc}((0, T) \times \mathbf{R}^d)$, and thus precompact in the space $L^1_{loc}((0, T) \times \mathbf{R}^d)$ by the compact Sobolev embedding. The precompactness in $L^2_{loc}((0, T) \times \mathbf{R}^d)$ follows from the precompactness in $L_{loc}^1((0, T) \times \mathbf{R}^d)$ and the bound $\Vert n^\tau \Vert_{L^{\infty}(\mathbf{R}^d)} \le \Vert n_0 \Vert_{L^{\infty}(\mathbf{R}^d)}$. We also observe that in both cases $D > 0$ and $D = 0$ the gradients $\nabla n^{\tau}_D$ are weakly-$*$ precompact in $L^{\infty}((0, T), \mathcal{M}_{loc}(\mathbf{R}^d))$.

Step 5. Conclusion.
\\
Following the arguments for \cite[Lemma 3.4]{Jacobs2022}
and \cite[Lemma 3.5]{Jacobs2022} we obtain, for each $D \ge 0$, the precompactness of $\{\rho_D^\tau\}_{\tau > 0}$ in $L^{1}((0, T) \times \mathbf{R}^d)$ and the weak precompactness of $\{p_D^{\tau}\}_{\tau > 0}$ and $\{\nabla p_D^{\tau}\}_{\tau > 0}$ in $L^{2}((0, T) \times \mathbf{R}^d)$. The only difference is that the needed bounds may depend on $D$. We now fix $D \ge 0$, we take a sequence $\tau_j \to 0$ and, by what we just proved, we can assume that there exist $\tilde{\rho} \in L^{1}((0, T) \times \mathbf{R}^d)$, $\tilde{p} \in H^1((0, T) \times \mathbf{R}^d)$ and $\tilde{n} \in L^{\infty}((0, T) \times \mathbf{R}^d) \cap L^2((0, T), L^2_{loc}(\mathbf{R}^d))$ with $\nabla \tilde{n} \in L^{\infty}((0, T), \mathcal{M}_{loc}(\mathbf{R}^d))$ such that, as $j \to +\infty$
\begin{align*}
\rho_{D}^{\tau_j} &\to \tilde{\rho}&&\text{strongly in}\ L^1((0, T) \times \mathbf{R}^d)
\\ p_D^{\tau_j} &\xrightharpoonup{} \tilde{p}&&\text{weakly in}\ L^2((0, T) \times \mathbf{R}^d)
\\ \nabla p_D^{\tau_j}&\xrightharpoonup{} \nabla \tilde{p}&&\text{weakly in}\ L^2((0, T) \times \mathbf{R}^d)
\\ n_D^{\tau_j} &\to \tilde{n}&&\text{strongly in}\ L^2((0, T), L^2_{loc}(\mathbf{R}^d))
\\ \nabla n_D^{\tau_j} &\to \nabla \tilde{n}&&\text{in the sense of distributions in}\ [0, T) \times \mathbf{R}^d.
\end{align*}
Arguing as in \cite[Proposition 3.6]{Jacobs2022} one can show that the triple $(\tilde{n}, \tilde{\rho}, \tilde{p})$ is a weak solution to \eqref{eq:model_eqnts} in the sense of Definition~\ref{eq:weak_solutions}. In particular, by the uniqueness part of Theorem~\ref{thm:ex_uniq} we infer that $(\tilde{n}, \tilde{\rho}, \tilde{p}) = (n_D, \rho_D, p_D)$.
\end{proof}
\subsection{Strong convergence of $\nabla p_D$}
Before entering the proof of Proposition~\ref{prop:strong_conv_pressure} we need two results: Lemma~\ref{lem:convergence_nut} improves the convergence of the nutrients $n_D$ to strong convergence in $L^2((0,T), L^2_{loc}(\mathbf{R}^d))$, while Lemma~\ref{lem:var_int_approx} gives a variational interpretation for the approximate pressure variable obtained by using Scheme \ref{item:second_scheme}.
\begin{lemma}\label{lem:convergence_nut}
Let $n_0 \in L^{\infty}(\mathbf{R}^d)$ such that $\nabla n_0 \in L^2_{loc}(\mathbf{R}^d)$, and let $\rho_0 \in BV(\mathbf{R}^d)$. For any $D > 0$ denote by $(n_D, \rho_D, p_D)$ the unique weak solution to \eqref{eq:model_eqnts} in the sense of Definition~\ref{eq:weak_solutions} with initial values $(n_0, \rho_0)$. Denote by $(n, \rho, p)$ the unique weak solution to \eqref{eq:model_eqnts} with $D=0$ and same initial values. Then for every $T > 0$ and every $R>0$ we have
\begin{equation}
\lim_{D \to 0} \Vert n_D - n \Vert_{L^2((0, T) \times B_R)} = 0.
\end{equation}
\end{lemma}
\begin{remark}\label{rem:gronwall}
Before entering the proof of Lemma~\ref{lem:convergence_nut} let us recall the following version of Gronwall's inequality: if $\alpha \in W^{1,1}((0, T))$ is a non-negative function satisfying for a.e.\ $t \in (0, T)$
\begin{equation}
\dot{\alpha}(t) \le c\alpha(t) + \beta(t),
\end{equation}
for a constant $c > 0$ and an integrable function $\beta \in L^1((0,T))$, then for every $t \in (0, T)$
\begin{equation}\label{eq:grw}
\alpha(t) \le e^{ct}\left( \alpha(0) + \int_0^t e^{-cs} \beta(s) ds\right).
\end{equation}
Note that $\beta$ can also assume negative values. 
\end{remark}
\begin{proof}[Proof of Lemma~\ref{lem:convergence_nut}]
We fix $T > 0$ and $R > 0$ as in the statements. We need to prove that
\begin{align*}
\lim_{D \to 0} n_D = n,\quad \text{in}\ L^2((0, T) \times B_R).
\end{align*}
To this aim, we pick a sequence of diffusion parameters $\{D_j\}_{j \in \mathbf{N}}$ such that $D_j \to 0$ as $j\to +\infty$ and we show that, up to the extraction of a subsequence,
\begin{align}\label{eq:claim_on_subs}
\lim_{j \to +\infty} n_{D_{j}} = n,\quad \text{in}\ L^2((0, T) \times B_R).
\end{align}
Step 1. We claim that
\begin{equation}\label{eq:this_strong_conv}
\sup_{j \in \mathbf{N}} \Vert \sqrt{D_{j}} \nabla n_{D_{j}} \Vert_{L^2((0, T) \times B_{4R})} < +\infty.
\end{equation}
To show this we pick a cut-off function $\eta \in C^{\infty}_c(\mathbf{R}^d)$ with $\mathbf{1}_{B_{4R}} \le \eta \le \mathbf{1}_{B_{8R}}$,we multiply the $n_{D_{j}}$ equation by $n_{D_{j}}\eta^2$ and we integrate in space to get
\begin{align*}
\frac{d}{dt}\frac{1}{2}\Vert n_{D_{j}}\eta \Vert^2_{L^2(B_{4R})} +\int_{B_{4R}} |\sqrt{D_{j}}\nabla n_{D_{j}}|^2 \eta^2 dx \le C\Vert n_0 \Vert_{L^{\infty}(\mathbf{R}^d)} \int_{B_{8R}}|\nabla n_{D_j}|dx.
\end{align*}
Thanks to Item \eqref{item:bound_grad_nut} in Theorem~\ref{thm:conv_D_zero} the right-hand side of the previous inequality is bounded by a constant $C$ non depending on $j$, thus integrating in time we obtain
\begin{equation}
\int_0^T \int_{B_{4R}} |\sqrt{D_j}\nabla n_{D_j}|^2 dx dt \le CT + \Vert n_0 \Vert_{L^2(B_{8R})}^2,
\end{equation}

Step 2. We claim that
\begin{equation}
D_j \Delta n_{D_j} \xrightharpoonup{} 0\quad \text{weakly in}\ L^2(B_{2R}).
\end{equation}
In view of Step 1, it is easy to see that $D_j \Delta n_{D_j}$ converges to zero in the sense of distributions. To conclude, it is thus sufficient to show that the sequence $\{D_j \Delta n_{D_j}\}_{j \in \mathbf{N}}$ is bounded in $L^2(B_{2R})$. To show this, we let $\eta \in C^{\infty}_c(\mathbf{R}^d)$ such that $\mathbf{1}_{B_{2R}} \le \eta \le \mathbf{1}_{B_{4R}}$. We multiply the equation for $n_{D_j}$ by $\partial_t n_{D_j} \eta^2$ and we integrate in space to get for a.e.\ $t \in (0, T)$
\begin{align*}
&\begin{aligned}
\Vert \partial_t n_{D_j} \eta \Vert_{L^2(B_{4R})}^2 + \int_{B_{4R}} D_j \nabla n_{D_j} \partial_t \nabla n_{D_j} \eta^2 dx = &-\int_{B_{4R}} \rho_{D_j} n_{D_j} \partial_t n_{D_j} \eta^2 dx
\\ & - 2\int_{B_{4R}} D_{j} \nabla n_{D_j} \partial_t n_{D_j} \nabla \eta \eta dx
\end{aligned}
\\ &\begin{aligned}
\phantom{\Vert \partial_t n_{D_j} \eta \Vert_{L^2(B_{2R})}^2 + \int_{B_{4R}} D_j \nabla n_{D_j} \partial_t \nabla n_{D_j} \eta^2 dx} \le &\frac{1}{2} \Vert \partial_t n_{D_j} \eta \Vert_{L^2(B_{4R})}^2 + C \Vert \rho_{D_j} n_{D_j} \Vert_{L^2(B_{4R})}^2 
\\ & + C \Vert D_j \nabla n_{D_j} \Vert_{L^2(B_{4R})}^2,
\end{aligned}
\end{align*}
where in the second line we used Young's inequality twice. We can thus rearrange terms so that
\begin{align}
&\begin{aligned}
\frac{1}{2}\Vert \partial_t n_{D_j} \Vert_{L^2(B_{4R})}^2 + \frac{d}{dt} \Vert \sqrt{D_j} \nabla n_{D_j} \eta \Vert_{L^2(B_{4R})}^2 \le &C \Vert \rho_{D_j} n_{D_j} \Vert_{L^2(B_{4R})}^2 
\\ &+ C \Vert D_j \nabla n_{D_j} \Vert_{L^2(B_{4R})}^2.
\end{aligned}
\end{align}
Integrating the previous inequality in time we obtain
\begin{align}
&\begin{aligned}
\frac{1}{2}\int_0^T \Vert \partial_t n_{D_j} \eta \Vert_{L^2(B_{4R})}^2 dt + \Vert \sqrt{D_j} \nabla n_{D_j}(T)\eta \Vert_{L^2(B_{4R})}^2 \le & \frac{1}{2}  \Vert \sqrt{D_j} \nabla n_0 \Vert_{L^2(B_{4R})}^2
\\ &+ CTR^d \Vert n_0 \Vert_{L^\infty(\mathbf{R}^d)}^2  
\\ & + C \int_0^T \Vert D_{j} \nabla n_{D_j} \Vert_{L^2(B_{4R})}^2dt.
\end{aligned}
\end{align}
In view of Step 1 and recalling that $\mathbf{1}_{B_{2R}} \le \eta$ we thus get
\begin{equation}
\sup_{j \in \mathbf{N}} \int_0^T \Vert \partial_t n_{D_j} \Vert_{L^2(B_{2R})}^2 dt < +\infty,
\end{equation}
which yields the claim by exploiting that $D_{j}\Delta n_{D_j} = \partial_t n_{D_j} + \rho_{D_j} n_{D_j}$.

Step 3. We prove that
\begin{equation}\label{eq:claim_larger_rad}
\lim_{j \to +\infty} \max_{t \in [0, T]} \Vert n_{D_{j}}(\cdot, t) - n(\cdot, t) \Vert_{L^2(B_{R})} = 0,
\end{equation}
which clearly implies \eqref{eq:claim_on_subs}.
\\
For this we subtract the $n$ equation from the $n_{D_j}$ equation, and we multiply the resulting equation by $(n_{D_j} - n)\eta^2$, where $\eta \in C^{\infty}_c(\mathbf{R}^d)$ is a cut-off function such that $\mathbf{1}_{B_R} \le \eta \le \mathbf{1}_{B_{2R}}$.  We then integrate in space to get
\begin{align}
&\begin{aligned}
\frac{1}{2}\frac{d}{dt} \Vert (n_{D_j} - n) \eta \Vert_{L^2(B_{2R})}^2 = &-D_j\int_{B_{2R}} |\nabla n_{D_j}|^2 \eta^2 dx - \int_{B_{2R}} D_{j} \Delta n_{D_j} n \eta^2 dx
\\ & -\int_{B_{2R}} \rho_{D_j}(n_{D_j} - n)^2 \eta^2 dx + \int_{B_{2R}} n(\rho_{D_j} - \rho)(n_{D_j}-n)\eta^2 dx
\\ & +2\int_{B_{2R}} D_j \nabla n_{D_j} \cdot \nabla \eta  \eta n_{D_j}dx
\end{aligned}
\\ &\begin{aligned}
\phantom{\frac{1}{2}\frac{d}{dt} \Vert (n_{D_j} - n) \eta \Vert_{L^2(B_{2R})}^2} \le &-\int_{B_{2R}}D_{j} \Delta n_{D_j} n\eta^2 dx + \Vert (n_{D_j} - n)\eta \Vert_{L^2(B_{2R})}^2 
\\ &+ C \int_{B_{2R}} |\rho_{D_j} - \rho| dx + C\sqrt{D_j} (\Vert \sqrt{D_j} \nabla n_{D_j} \Vert_{L^2(B_{2R})}^2 + 1).
\end{aligned}
\end{align}
We now apply Gronwall's inequality in the form of Remark~\ref{rem:gronwall}, with the choices $\alpha = \frac{1}{2}\frac{d}{dt} \Vert (n_{D_j} - n) \eta \Vert_{L^2(B_{2R})}^2$, $c = 2$ and $\beta = -\int_{B_{2R}}D_{j} \Delta n_{D_j} n\eta^2 dx + C \int_{B_{2R}} |\rho_{D_j} - \rho| dx + C\sqrt{D_j} (\Vert \sqrt{D_j} \nabla n_{D_j} \Vert_{L^2(B_{2R})}^2 + 1)$ to get
\begin{align}
\max_{t \in [0, T]} \Vert (n_{D_j} - n) \Vert_{L^2(B_R)}^2 \le &-e^{2T}\int_0^T \int_{B_{2R}} D_{j} \Delta n_{D_j} (e^{-2t} \eta^2 n) dx dt 
\\ &+ Ce^{2T} \int_0^T\int_{B_{2R}} |\rho_{D_j} - \rho| dxdt
\\ &+ Ce^{2T}\sqrt{D_j} (\Vert \sqrt{D_j} \nabla n_{D_j} \Vert_{L^2((0, T) \times B_{2R})}^2 + T)
\end{align}
We conclude by observing that by what we proved in Step 1 and Step 2, and by Item \eqref{item:rho_conv} in Theorem~\ref{thm:conv_D_zero} the right-hand side of the previous inequality converges to zero as $j \to +\infty$.
\end{proof}
\begin{lemma}\label{lem:var_int_approx}
Let $n_0 \in L^\infty(\mathbf{R}^d)$ with $\nabla n_0 \in L^2_{loc}(\mathbf{R}^d)$ and let $\rho_0 \in BV(\mathbf{R}^d)$ be compactly supported with $0 \le \rho_0 \le 1$. For $D \ge 0$, let $\{(n^\tau_D, \rho^\tau_D, p^\tau_D)\}_{\tau > 0}$ be the family of approximate solutions obtained using Scheme \ref{item:second_scheme} with initial values $(n_0, \rho_0)$. Then we have for all $\tau > 0$, for all $D \ge 0$ and all $t > 0$
\begin{equation}\label{eq:variational_ineq_approx}
\int_{\mathbf{R}^d} \nabla \xi \cdot \nabla p^{\tau}_D(x,t)dx \le \int_{\mathbf{R}^d} \xi n_D^\tau(x, t-\tau)\rho_D^\tau(x, t-\tau)\mu_D^{\tau}(x, t-\tau)dx\quad \forall \xi \in H_{\rho_D^\tau(t)},
\end{equation}
where $H_{\rho_D^\tau(t)}$ is defined as in \eqref{eq:a_def_for_set} in Section \ref{sec:mainres}.
\end{lemma}
\begin{proof}
For ease of notation, for every $k \in \mathbf{N}$ we define
\begin{equation}
n^{k, \tau}_D := n^{\tau}_D(k\tau).
\end{equation}
We let $k \in \mathbf{N}$ such that $t \in [k\tau, (k+1)\tau)$, then by definition we have $\nabla p_D^{\tau}(t) = \nabla p_D^{k, \tau}$. We also recall that $T^k:\mathbf{R}^d \to \mathbf{R}^d$ defined by
\begin{equation}
T^k := (Id + \tau \nabla p_D^{k+1, \tau})
\end{equation}
is the optimal transport map from $\rho_D^{ k+1, \tau}$ to $\mu_D^{k, \tau}$. We define the interpolation maps $T_t := tT^k+(1-t)Id$ and we define $\gamma_t:= (T_t)_{\#}\rho_D^{k+1, \tau}$. Then $\gamma_t$ is a measure on $\mathbf{R}^d$, absolutely continuous with respect to the Lebesgue measure, with density
\begin{align}
\gamma_t = \frac{\rho_D^{ k+1, \tau} \circ T_t^{-1}}{|\operatorname{det}\nabla T_t|}&\le \left(t \frac{\rho_D^{k+1, \tau}}{|\operatorname{det}\nabla T^k|} + (1-t) \frac{\rho_D^{k+1, \tau}}{|\operatorname{det}Id|}\right) \circ T_t^{-1},\label{eq:after_convex}
\end{align}
where in the inequality we used the fact that $|\operatorname{det}M|^{-1}$ is convex on the space of positive-definite matrices. Observe that
\begin{equation}
\frac{\rho_D^{k+1, \tau}}{|\operatorname{det}\nabla T^k|}(x) = \mu_D^{k, \tau}(T^k(x)) \le \bigg(1+\tau \rho_D^{k, \tau}(T^k(x))n_D^{k, \tau}(T^k(x))\bigg).
\end{equation}
In particular, inserting back into \eqref{eq:after_convex} this yields
\begin{equation}
\gamma_t(x) \le 1 +  t\tau \rho_D^{k, \tau}(T^k\circ T_t^{-1}(x))n_D^{k, \tau}(T^k \circ T_t^{-1}(x)).
\end{equation}
We now take $\xi \in H_{\rho_D^\tau(t)}$ and evaluate
\begin{align*}
&\int_{\mathbf{R}^d} (\xi(T_t(x)) - \xi(x))\rho_{D}^{k+1, \tau}(x) dx 
\\ &\le  \int_{\mathbf{R}^d} \xi(x) (1 +  t\tau \rho_D^{k, \tau}(T^k\circ T_t^{-1}(x))n_D^{k, \tau}(T^k\circ T_t^{-1}(x))) dx- \int_{\mathbf{R}^d}\xi(x)\rho_{D}^{k+1, \tau}(x) dx
\\ &\le t\tau\int_{\mathbf{R}^d} \xi(x) \rho_D^{k, \tau}(T^k \circ T_t^{-1}(x))n_D^{k, \tau}(T^k \circ T_t^{-1}(x))) \rho_{D}^{k+1, \tau} dx,
\end{align*}
where in the last line we used that $\xi (1-\rho_{D}^{k+1, \tau}) = 0$. We now divide the previous inequality by $t$ and let $t \to 0$ to obtain
\begin{align*}
\int_{\mathbf{R}^d} \nabla \xi \cdot \nabla p_D^{k+1, \tau} \rho_D^{k+1, \tau} dx \le \int_{\mathbf{R}^d} \xi \rho_D^{k, \tau} n_{D}^{k, \tau} \mu_D^{k, \tau} dx.
\end{align*}
Using again that $\xi(1-\rho_{D}^{k+1, \tau}) = 0$ yields the claim.
\end{proof}
\begin{proof}[Proof of Proposition~\ref{prop:strong_conv_pressure}]
We first claim that for every $\xi \in H^1(\mathbf{R}^d)$ and every $\delta > 0$ we have
\begin{align}
&\int_t^{t+\delta} \int_{\mathbf{R}^d}\nabla p_D \cdot (\nabla p_D - \nabla \xi) dx ds\label{eq:claim_without_restrict}
\\ &\le \int_{\mathbf{R}^d} (\rho_D(t+\delta) - \rho_D(t))\xi dx+\int_t^{t+\delta}\int_{\mathbf{R}^d} n_D \rho_D(p_D \rho_D - \xi) dx ds.\nonumber
\end{align}
To prove \eqref{eq:claim_without_restrict} we can assume that $\xi \in C^{\infty}_c(\mathbf{R}^d)$, because $C^{\infty}_c(\mathbf{R}^d)$ is dense in $H^1(\mathbf{R}^d)$. We let $(n_D^\tau,\rho_D^\tau, p_D^\tau)$ be the approximations obtained by using the Scheme \ref{item:second_scheme}. For ease of notation for every $k \in \mathbf{N}$ we define
\begin{equation}
n^{\tau, k}_D := n^{\tau}_D(k\tau).
\end{equation}
We fix $k \in \mathbf{N}$ and we recall that $T^k(x) := x + \tau \nabla p^{\tau, k+1}_D$ is the optimal transport map between $\rho^{\tau, k+1}_D$ and $\mu^{\tau, k}_D$.  We observe that by definition of $T^k$, by a Taylor expansion of $\xi$ and using the fact that $p^{\tau, k+1} \in H_{\rho^{\tau, k+1}_D}$
\begin{align*}
\phantom{\int_{\mathbf{R}^d} (\rho^{\tau, k+1}_D - \mu^{\tau, k}_D)\xi dx}
&\begin{aligned}
\mathllap{\int_{\mathbf{R}^d} (\rho^{\tau, k+1}_D - \mu^{\tau, k}_D)\xi dx} = \int_{\mathbf{R}^d}(\xi(x) - \xi(T^k(x)))\rho^{k+1, \tau}_D dx
 \end{aligned}
\\ &\begin{aligned}
\mathllap{} = &\int_{\mathbf{R}^d} \nabla \xi (x) \cdot (x-T^k(x))\rho^{k+1, \tau}_D dx 
\\ &+ O\left(\Vert D^2\xi\Vert_{L^{\infty}}\int_{\mathbf{R}^d}|x - T^k(x)|^2\rho^{k+1, \tau}_D dx \right)
\end{aligned}
\\ &\begin{aligned}
\mathllap{} = &-\tau\int_{\mathbf{R}^d} \nabla \xi \cdot \nabla p^{k+1, \tau}_D dx 
\\ &+ O\left(\Vert D^2\xi\Vert_{L^{\infty}}\int_{\mathbf{R}^d}|x - T^k(x)|^2\rho^{k+1,\tau}_D dx \right).
\end{aligned}
\end{align*}
We rewrite the previous identity as
\begin{align}\label{eq:first_id}
-\int_{\mathbf{R}^d} \nabla \xi \cdot \nabla p^{k+1, \tau}_D dx = &\frac{1}{\tau}\int_{\mathbf{R}^d}(\rho^{k+1, \tau}_D- \mu^{\tau, k}_D) \xi dx 
\\ &+O\left(\Vert D^2\xi\Vert_{L^{\infty}}\frac{1}{\tau}W^{2}_2(\rho^{k+1, \tau}_D, \mu^{k, \tau}_D) \right).\nonumber
\end{align}
We now use \eqref{eq:variational_ineq_approx} with $\xi = \rho^{k+1, \tau}_D$ to get
\begin{equation}\label{eq:second_id}
\int_{\mathbf{R}^d} \nabla p^{k+1, \tau}_D \cdot \nabla p^{k+1, \tau}_D dx \le \int_{\mathbf{R}^d} p^{k+1,\tau}_D n^{k, \tau}_D \rho^{k, \tau}_D \mu^{k, \tau}_D dx.
\end{equation}
We now sum \eqref{eq:first_id} and \eqref{eq:second_id} to get, using also the definition of $\mu^{k, \tau}_D$
\begin{align*}
\int_{\mathbf{R}^d} \nabla p^{k+1, \tau}_D \cdot (\nabla p^{k+1, \tau}_D - \nabla \xi)dx \le & \frac{1}{\tau} \int_{\mathbf{R}^d} (\rho^{k+1, \tau}_D - \rho^{k, \tau}_D)\xi dx 
\\ & - \int_{\mathbf{R}^d} \rho^{k, \tau}_D n^{k, \tau}_D \xi dx
\\ &+\int_{\mathbf{R}^d} p^{k+1, \tau}_D n^{k, \tau}_D \rho^{k, \tau}_D \mu^{k, \tau}_D dx
\\ &+ O\left(\Vert D^2\xi\Vert_{L^{\infty}}\frac{1}{\tau}W^{2}_2(\rho^{k, \tau}_D, \mu^{k+1, \tau}_D) \right).
\end{align*}
In other words, if we fix $t > 0$ and we integrate over $(t, t+\delta)$ we obtain
\begin{align}
&\begin{aligned}
\int_{t}^{t+\delta}\int_{\mathbf{R}^d} \nabla p^{\tau}_D \cdot (\nabla p^{\tau}_D - \nabla \xi)dx ds
\end{aligned}
\\ &\begin{aligned}
\le & \int_{\mathbf{R}^d} (\rho^{\tau}_D(t+\delta) - \rho^{\tau}_D(t))\xi  dx +\int_{t}^{t+\delta}\int_{\mathbf{R}^d} n^{\tau}_D(t-\tau) \rho^{\tau}_D(t-\tau) (p^{\tau}_D \mu^{\tau}_D(t-\tau) - \xi) dx ds
\\ &+ O\left(\Vert D^2\xi\Vert_{L^{\infty}}\sum_{0 \le k \le \left[\frac{T}{\tau}\right]}W^{2}_2(\rho^{k+1, \tau}_D, \mu^{k, \tau}_D) \right). \label{eq:before_tau_zero}
\end{aligned}
\end{align}
Now observe that 
\begin{equation}
\frac{\tau}{2}\Vert \nabla p^{k+1, \tau}_D\Vert_{L^2(\mathbf{R}^d)}^2 = \frac{1}{2\tau}W_2^2\left( \rho^{k+1, \tau}_D, \mu^{k, \tau}_D \right),
\end{equation}
in particular, this yields
\begin{equation}
\sum_{0 \le k \le \left[\frac{T}{\tau}\right]} \frac{1}{2} W_2^2\left( \rho^{k+1, \tau}_D, \mu^{k, \tau}_D \right) = O(\tau).
\end{equation}
We now let $\tau \to 0$ in \eqref{eq:before_tau_zero}: for the left hand side we use the weak convergence of $\nabla p_D^\tau$ to $\nabla p_D$ in $L^2((t, t+\delta) \times \mathbf{R}^d)$ (Item \eqref{item:nabla_p_conv_second} in Theorem~\ref{thm:estimate_n_second_conv}), for the first right hand side term we use the pointwise $L^1(\mathbf{R}^d)$ convergence of $\rho_D^\tau(\cdot)$ (Item \eqref{item:rho_conv_second} in Theorem~\ref{thm:estimate_n_second_conv}). For the second right hand side term we use the weak convergence of $p_D^\tau$ in $L^2((t, t+\delta) \times \mathbf{R}^d)$ (Item \eqref{item:p_conv_second} in Theorem~\ref{thm:estimate_n_second_conv}), and the strong convergences of  $\rho_D^\tau$ and $n_D^\tau$ (items \eqref{item:rho_conv_second} and \eqref{item:n_conv_second} in Theorem~\ref{thm:estimate_n_second_conv}). We thus obtain \eqref{eq:claim_without_restrict}.

Now, if we additionally assume that $\xi \in H^1_{\rho_D(t)}$ we must have that $(\rho_D(t+\delta) - \rho_D(t))\xi = (\rho_D(t+\delta) - 1)\xi \le 0$, thus
\begin{align}
&\int_t^{t+\delta} \int_{\mathbf{R}^d}\nabla p_D \cdot (\nabla p_D - \nabla \xi) dx ds \le \int_t^{t+\delta}\int_{\mathbf{R}^d} n_D \rho_D(p_D \rho_D - \xi) dx ds.\nonumber
\end{align}
We divide the previous inequality by $\delta$ and let $\delta$ to zero to obtain \eqref{eq:variational_ineq_pressure}. 

It remains to prove that $\nabla p_D$ converges strongly to $\nabla p$ in $L^2((0, T) \times \mathbf{R}^d)$ as $D \to 0$. By Item \eqref{item:weak_conv_grad} in Theorem~\ref{thm:conv_D_zero} we already know that $\nabla p_D$ converges to $\nabla p$ weakly in $L^2((0, T) \times \mathbf{R}^d)$. Since this is a Hilbert space, the strong convergence follows once we prove that
\begin{equation}
\lim_{D \to 0} \int_0^T \Vert \nabla p_D(t) \Vert_{L^2(\mathbf{R}^d)}^2 dt =  \int_0^T \Vert \nabla p(t) \Vert_{L^2(\mathbf{R}^d)}^2 dt.
\end{equation}
To show this, we preliminary observe that
\begin{equation}\label{eq:convergence_triple}
\lim_{D \to 0} \int_0^T \int_{\mathbf{R}^d} p_Dn_D\rho_D dx dt = \int_0^T \int_{\mathbf{R}^d} p n \rho dx dt.
\end{equation}
Indeed, this follows from the weak $L^2((0, T) \times \mathbf{R}^d)$ convergence of $p_D$ (Item \eqref{item:weak_conv_grad} in Theorem~\ref{thm:conv_D_zero}), the strong $L^1((0, T) \times \mathbf{R}^d)$ convergence of $\rho_D$ to $\rho$ (Item \eqref{item:rho_conv} in Theorem~\ref{thm:conv_D_zero}), the strong $L^2((0, T), L^2_{loc}(\mathbf{R}^d))$ convergence of the nutrients  in Lemma~\ref{lem:convergence_nut} and the fact that, thanks to Lemma~\ref{lem:bound_supp} and Theorem~\ref{thm:estimate_n_second_conv}, we have
\begin{equation*}
\bigcup_{D \ge 0} \bigcup_{0 \le t \le T} \operatorname{supp}(\rho_D(t)) \subset B_{R(T)},
\end{equation*}
where $R(T)$ is as in \eqref{eq:def_rad} in Lemma~\ref{lem:bound_supp}. We now choose $\xi = 2p(x, t)$ in \eqref{eq:variational_ineq_approx} for $D = 0$ and we get, using also the weak lower-semicontinuity of the $L^2$-norm and \eqref{eq:convergence_triple}
\begin{align*}
\liminf_{D\to 0} \int_0^T \Vert \nabla p_D(t)\Vert_{L^2(\mathbf{R}^d)}^2dt &\ge  \int_0^T \Vert \nabla p(t)\Vert_{L^2(\mathbf{R}^d)}^2dt
\\ &\ge \int_0^T \int_{\mathbf{R}^d} pn\rho dx dt
\\ &= \limsup_{D \to 0} \int_0^T \int_{\mathbf{R}^d} p_Dn_D\rho_D dx dt
\\ &\ge \limsup_{D \to 0} \int_0^T \Vert \nabla p_D(t)\Vert_{L^2(\mathbf{R}^d)}^2 dt,
\end{align*}
where in the last line we used \eqref{eq:variational_ineq_pressure} with $\xi = 0$.
\end{proof}
\subsection{Hausdorff convergence of the tumor patches: Proof of Theorem~\ref{thm:conv_patches}} 
The purpose of this subsection is to give a proof of Theorem~\ref{thm:conv_patches}. Recalling the definition of Hausdorff distance, we need to show, under the assumptions of Theorem~\ref{thm:conv_patches}, that the following two statements hold true
\begin{align}
&\lim_{D\to 0} \sup_{x \in \Gamma(t)} d(x, \Gamma_D(t)) = 0,\quad t \in (\hat{t}, T),\label{eq:easy_part}
\\ & \lim_{D \to 0} \sup_{x \in \Gamma_D(t)} d(x, \Gamma(t)) = 0,\quad t \in (\hat{t}, T).\label{eq:hard_part}
\end{align}
We first show that \eqref{eq:easy_part} holds true, this is the content of the following proposition. We warn the reader that hereafter $\omega_d$ is a constant denoting the Lebesgue measure of the unit ball in $\mathbf{R}^d$.
\begin{proposition}\label{prop:easy_part}
Under the assumptions of Theorem~\ref{thm:conv_patches}, we have that for every $t \in [\hat{t}, T]$
\begin{equation}
\lim_{D\to 0} \sup_{x \in \Gamma(t)} d(x, \Gamma_D(t)) = 0.
\end{equation}
\end{proposition}
\begin{proof}[Proof of Proposition~\ref{prop:easy_part}]
Fix $t \in [\hat{t}, T]$ and take an arbitrary $\epsilon > 0$, we have to show that for $D$ sufficiently small and for every $x \in \Gamma(t)$
\begin{equation}\label{eq:eps_d_claim}
d(x, \Gamma_D(t)) < \epsilon.
\end{equation}
To this aim, fix $\delta < 1$ to be determined later. We observe that by \cite[Lemma 4.6]{Jacobs2022} we have $\{\rho(t) = 1\} = \{p(t) > 0\}$. Since $n_0 \ge \lambda > 0$, it is also easy to see that $\operatorname{Int}\{\rho(t) = 1\} = \operatorname{Int}\{p(t) > 0\}$. In particular, $\Gamma(t) = \partial\{\rho(t) = 1\}$. We observe that since the boundary $\partial\{\rho(t) = 1\}$ is $C^1$ uniformly in time, there exist $\alpha \in (0, \frac{\pi}{2}]$ and $\beta > 0$ such that for every $t \in [\hat{t}, T]$ the set $\{\rho(t) = 1\}$ satisfies the uniform exterior cone property with parameters $\alpha$ and $\beta$. This means that for every point $x \in \partial\{\rho(t) = 1\}$ we may find a unit vector $v$ such that
\begin{equation}
    x + K_{\alpha}(v) \cap B_{\beta}(x) \subset \mathbf{R}^d\setminus\{p(t) > 0\},
\end{equation}
where $K_{\alpha}(v) := \{ z \in \mathbf{R}^d: z \cdot v \ge |z|\operatorname{cos}(\alpha)\}$. In particular, for any $x \in \partial\{\rho(t) = 1\}$, if $\delta \le \beta$ we have
\begin{equation}\label{eq:notfillingball}
    |\{\rho(t) = 1\} \cap B_{\delta}(x)| \le \left( 1-\frac{\alpha}{2\pi} \right) \omega_d \delta^d.
\end{equation}
We now assume that $\delta \le \beta$. We define
\begin{equation}
c_0 := \inf_{y\in \overline{\{p(t) > 0\}},\, d(y, \Gamma(t)) \le \delta} \dashint_{B_{\delta}(y)} p(t, z) dz.
\end{equation}
We observe that $c_0 > 0$, because the function
\begin{equation}
y \mapsto \dashint_{B_{\delta}(y)} p(t, z) dz
\end{equation}
is continuous and the set $\{ y \in \overline{\{p(t) > 0\}}:\, d(y, \Gamma(t)) \le \delta\}$ is compact. By Item \eqref{item:pressure_conver} in Remark~\ref{rem:density_smooth} and by Item \eqref{item:rho_conv} in Theorem~\ref{thm:conv_D_zero} we can select $\overline{D} > 0$ such that for every $D \le \overline{D}$ 
\begin{align}
\Vert p_D - p \Vert_{L^1((0, T) \times \mathbf{R}^d)} &\le c_0 \omega_d \frac{\delta^d}{2},\label{eq:choice_pressure} 
\\ \Vert \rho_D(t) - \rho(t) \Vert_{L^1(\mathbf{R}^d)} &\le \omega_d \frac{\alpha}{4\pi}\delta^d.\label{eq:choice_patch}
\end{align}
Let $x \in \Gamma(t)$. For each $D \le \overline{D}$ we distinguish three cases:
\begin{enumerate}
\item If $x \in \Gamma_D(t)$ then by definition $d(x, \Gamma_D(t)) = 0$ and \eqref{eq:eps_d_claim} holds trivially.
\item If $x \in \{p_D(t)>0\}$, using \eqref{eq:notfillingball} and \eqref{eq:choice_patch}
we have
\begin{equation}
\Vert \rho_D(t) \Vert_{L^1(B_{\delta}(y))} \le \left( 1-\frac{\alpha}{4\pi}\right)\omega_d \delta^d.
\end{equation}
Since $\rho_D(t) \in \{0,1\}$ this implies that there exists $z \in B_{\delta}(y) \setminus \{p_D(t) > 0\}$. Define
\begin{equation}
r := \sup\left\{ r \in [0, 1]:\ z + r(x-z) \not \in \{p_D(t) > 0\} \right\},
\end{equation}
it follows that $z + r(x-z) \in \Gamma_D(t)$ and
\begin{equation}
|x-(z+r(x-z))| \le |1-r||x-z| \le 2\delta.
\end{equation}
Thus $d(x, \Gamma_D(t)) \le 2\delta < \epsilon$ provided we choose $\delta$ small enough.
\item If $x \not \in \overline{\{p_D(t)>0\}}$, then using \eqref{eq:choice_pressure} we get
\begin{equation}
\dashint_{B_{\delta}(y)} p_D(t, z)dz \ge \frac{c_0}{2} > 0,
\end{equation}
thus there exists $z \in B_\delta(y) \cap \{p_D(t) > 0\}$ and a similar reasoning as in the previous case yields \eqref{eq:eps_d_claim} if $\delta$ is small enough.\qedhere
\end{enumerate}
\end{proof}
We will now consider the harder part \eqref{eq:hard_part}. The first step is Proposition~\ref{prop:inner_control}, which allows us to control the growth of the fingers -- or better, tubes -- of the $D>0$ patch inside the $D = 0$ patch. Before stating the result, let us introduce some notation. For a given $s > 0$ and for any time $t > 0$ we define
\begin{equation}
U_s(t) := \{ x \in \{p(t) > 0\}:\ d(x, \Gamma(t)) \ge s\}.
\end{equation}
We then have the following result.
\begin{proposition}\label{prop:inner_control}
With the assumptions in Theorem~\ref{thm:conv_patches} and  $r_{tub}$ as defined above, there exists $\delta_0 > 0$ such that  the following holds: for every $\delta < \delta_0$ there exists $D_0(\delta)>0$ such that whenever $D \le D_0$ we have 
\begin{equation}
U_{5\delta}(t) \subset \{p_D(t+\sqrt{\delta}) > 0\} \hbox{ for } t\in [\hat{t}, T-\sqrt{\delta}].
\end{equation}
\end{proposition}
\begin{proof}
The proof is divided into five steps.
\\
Step 1. We define
\begin{equation}
\kappa_T := \inf_{\hat{t} \le t \le T} \inf_{v \in \partial \Omega_t} |\nabla p(t, v)|,
\end{equation}
and we claim that if $\delta$ is small enough
\begin{equation}\label{eq:lower_bound_D_zero}
\inf_{t \in [\hat{t}, T]} \inf_{y \in U_{\delta}(t)} p(t, y) \ge\frac{\kappa_T}{2}\delta.
\end{equation}
To show this, let $y \in U_{\delta}(t)$, and let $x \in \partial\{p(t) > 0\}$ be any point on $\partial\{p(t) > 0\}$ such that $|y-x| = d(y, \partial\{p(t) > 0\})$. Since the boundary of $\{p(t) > 0\}$ is $C^1$ we have that $y-x$ is parallel to the inner normal $\frac{\nabla p}{|\nabla p|}(t, x)$ at $x$, thus, using also that $p(t, x) = 0$
\begin{align}
p(t, y) &= p(t, x) + \nabla p(t, x) \cdot (y - x) + o(\delta)
\\ &= 0 + |\nabla p(t, x)||y - x| + o(\delta)
\\ &\ge  |\nabla p(t, \pi(y, t))|\delta + o(\delta).
\end{align}
We have, if $\delta$ is chosen small enough,
\begin{equation}
p(t, y) \ge \kappa_T \delta +o(\delta) \ge \frac{\kappa_T}{2}\delta.
\end{equation}
which is \eqref{eq:lower_bound_D_zero}.

Step 2. We claim that for every $D > 0$
\begin{equation}\label{eq:lower_bound_n_claimed}
\inf_{x \in \mathbf{R}^d, t \in [0, T]} n_D(t, x) \ge e^{-T}\lambda, 
\end{equation}
where $\lambda$ is defined as in the statement of Theorem~\ref{thm:conv_patches}.
\\
To prove \eqref{eq:lower_bound_n_claimed} we work with the approximate nutrient variables $\{n_D^\tau\}_{\tau > 0}$ obtained by using Scheme \ref{item:first_scheme}. Recall that for every $k \in \mathbf{N}$, and every $\tau > 0$, we have
\begin{equation}
n_D^{k+1, \tau} = e^{\tau D \Delta}( n_D^{k, \tau}(1-\tau \rho_D^{k+1, \tau})).
\end{equation}
If we assume inductively that $n_D^{k, \tau} \ge \lambda (1-\tau)^k$ (which is true for $k = 0$ by assumption), then we easily get $n_D^{k+1, \tau} \ge \lambda(1-\tau)^{k+1}$. Thus for every $k \in \mathbf{N}$
\begin{equation}
n_D^{k, \tau} \ge \lambda(1-\tau)^k.
\end{equation}
For $\tau$ sufficiently small, the right-hand side is bounded from below by $\lambda e^{-k\tau}$. Thus for $k \le \left[\frac{T}{\tau}\right]$
\begin{equation}
n_D^{k, \tau} \ge e^{-T}\lambda.
\end{equation}
We now pick any $\varphi \in L^1((0, T) \times \mathbf{R}^d), \varphi \ge 0$, and observe that by the previous inequality and by using the weak-$*$ convergence of $n_D^\tau$ to $n_D$ in $L^{\infty}(\mathbf{R}^d)$ (which is proved in \cite[Proposition 3.6]{Jacobs2022}) we obtain
\begin{align*}
\int_0^T \int_{\mathbf{R}^d} \varphi(n_D - e^{-T}\lambda) dx dt = \lim_{\tau \to 0} \int_0^T \int_{\mathbf{R}^d} \varphi (n_D^\tau - e^{-T}\lambda) dx dt \ge 0.
\end{align*}
Since $\varphi \in L^1((0, T) \times \mathbf{R}^d)$ was arbitrary we infer \eqref{eq:lower_bound_n_claimed}.

Step 3. Fix $\gamma = \sqrt{\delta}$. We claim that for every $D > 0$ the function
\begin{equation}
w_D(t, x) := \frac{1}{\gamma}\int_t^{t+\gamma} p_D(s, x)ds,
\end{equation}
solves for every $t > 0$
\begin{equation}\label{eq:supersol}
-\Delta w_D(t) \ge -\frac{1}{\gamma}\quad \text{on}\ \mathbf{R}^d.
\end{equation}
\\
To see that \eqref{eq:supersol} holds, we integrate in time the $\rho_D$ equation from $t$ to $t+\gamma$ to get
\begin{equation}
-\Delta w_D(t) = \frac{1}{\gamma}\rho_D(t) + \frac{1}{\gamma}\int_t^{t+\gamma}\rho_D(s)n_D(s)ds - \frac{1}{\gamma}\rho_D(t+\gamma) \ge -\frac{1}{\gamma}.
\end{equation}

Step 4. We claim that there exists a constant $c_d$ depending only on the dimension $d$ such that for every $t \in [\hat{t}, T]$ and for every $D > 0$, if $x_0 \in U_{5\delta}(t)$ is such that 
\begin{align}
\begin{aligned}\label{eq:maximal_bound}
\dashint_{B_{2\delta}(x_0)} |w_D(t, y) - w(t, y)| dy \le c_d\frac{\kappa_T}{4}\delta
\end{aligned}
\end{align}
then, provided $\delta$ is small enough, we have
\begin{equation}\label{eq:control_harnack}
\inf_{B_{\delta}(x_0)} w_D(t, \cdot) \ge c_0,
\end{equation}
where we define $c_0 = \frac{c_d\kappa_T\delta}{8}$. In particular, we have that
\begin{equation}\label{eq:easycon}
B_{\delta}(x_0) \subset \{ p_D(t+\sqrt{\delta}) > 0\}.
\end{equation}
We first show that inclusion \eqref{eq:easycon} follows from \eqref{eq:control_harnack}. Indeed, observe that for any $x \in \mathbf{R}^d$, since the set $\{\rho_D(t, \cdot) = 1\}$ is expanding in time, we have
\begin{align*}
w_D(t, x)(1-\rho_D(t+\gamma, x)) &= \frac{1}{\gamma}\int_t^{t+\gamma} p_D(s, x)(1-\rho_D(t+\gamma, x)) ds 
\\ &\le \frac{1}{\gamma}\int_t^{t+\gamma}p_D(s, x)(1-\rho_D(s,x))ds = 0.
\end{align*}
This implies that $B_{\delta}(x_0) \subset \{w_D(t, \cdot) > 0\} \subset \operatorname{Int}\{\rho_D(t+\gamma, \cdot) = 1 \}$. In particular, we have
\begin{equation}
-\Delta p_D(t+\gamma) = n_D\quad \text{on}\ B_{\delta}(x_0).
\end{equation}
By what we proved in Step 2, we have that 
\begin{equation}
-\Delta p_D(t+\gamma) \ge e^{-T} \lambda\quad \text{on}\ B_{\delta}(x_0),
\end{equation}
so that the maximum principle implies that $p_D(t+\gamma)$ is strictly positive on $B_{\delta}(x_0)$.
\\
To show \eqref{eq:control_harnack} we use the weak Harnack's inequality \cite[Theorem 8.18]{Gilbarg2001} with $q = d+1$. Since $w_D$ satisfies \eqref{eq:supersol}, we have that there exists a constant $c_d$ depending only on the dimension $d$ such that
\begin{equation}
\inf_{B_{\delta}(x_0)} w_D(t, \cdot) + \frac{\delta^{2-\frac{2d}{q} + \frac{2d}{q}}}{\gamma} \ge c_d \dashint_{B_{2\delta}(x_0)} w_D(t, x) dx.
\end{equation}
In other words
\begin{equation}\label{eq:an_ineq}
\inf_{B_{\delta}(x_0)} w_D(t, \cdot) \ge c_d \dashint_{B_{2\delta}(x_0)} w_D(t, x) dx - \delta^{\frac{3}{2}}.
\end{equation}
Now we observe that due to \eqref{eq:lower_bound_D_zero} and \eqref{eq:maximal_bound} we have
\begin{align}
\dashint_{B_{2\delta(x_0)}} w_D(t, x) dx &= \dashint_{B_{2\delta}(x_0)} \frac{1}{\gamma}\int_t^{t+\gamma}p_D(s, x)ds dx
\\ &= \dashint_{B_{2\delta}(x_0)} \frac{1}{\gamma}\int_t^{t+\gamma}p(s, x)ds dx + \dashint_{B_{2\delta}(x_0)} \big(w_D(t, x) - w(t, x)\big) dx
\\ &\ge \frac{\kappa_T}{2}\delta - c_d\frac{\kappa_T}{4}\delta.
\end{align}
In particular, we get
\begin{equation}
\inf_{B_{\delta}(x_0)} w_D(t, \cdot) \ge \frac{c_d\kappa_T\delta}{8},
\end{equation}
provided $\delta$ is small enough, the smallness depending only on the dimension $d$.
\\
Step 5. Conclusion. By Item \eqref{item:pressure_conver} in Remark~\ref{rem:density_smooth} we can find $D_0(\delta) > 0$ such that for $D \le D_0$
\begin{align}
\begin{aligned}
\Vert p - p_D \Vert_{L^1((0, T)\times \mathbf{R}^d)} \le \frac{\omega_d c_d \kappa_T \delta^{d+1}\gamma}{4}.
\end{aligned}\label{eq:press_cond}
\end{align}
We now fix $D \le D_0$ and we observe that if $z \in U_{5\delta}(t)$ then 
\begin{equation}\label{eq:whatclaim}
\dashint_{B_{2\delta}(x_0)} |w_D(t, y) - w(t, y)|dy \le c_d\frac{\kappa_T}{4}\delta.
\end{equation}
Indeed, by \eqref{eq:press_cond} we get
\begin{align*}
\dashint_{B_{2\delta}(x_0)} \left| w_D(t, y) - w(t,y) \right|dy &\le \frac{1}{\gamma}\int_t^{t+\gamma} \dashint_{B_{2\delta}(x_0)} |p_D(s, y) - p(s, y)| dy ds
\\ & \le c_d\frac{\kappa_T}{4}\delta. 
\end{align*}
Since $z \in U_{5\delta}(t)$, we infer from Step 3 that $z \in \{p_D(t+\sqrt{\delta})>0\}$. Since $z \in U_{5\delta}(t)$ was arbitrary we get the claim.
\end{proof}

The second step in the proof of Theorem~\ref{thm:conv_patches} is the following proposition which allows us to control the growth of fingers outside the smooth patch. Hereafter, for $x_0 \in \mathbf{R}^d$ and $0 < r_1 < r_2$, we denote by $A_{r_1, r_2}(x_0)$ the annulus given by
\begin{equation}
A_{r_1, r_2}(x_0) := B_{r_2}(x_0) \setminus B_{r_1}(x_0)
\end{equation}
 
\begin{proposition}\label{prop:ext_cont}
Assume that $n_0 \in L^{\infty}(\mathbf{R}^d)$ is such that $\nabla n_0 \in L^2_{loc}(\mathbf{R}^d)$, and that $\rho_0~=~\chi_{\Omega_0} \in BV(\mathbf{R}^d)$. For every $D > 0$ let $(n_D, \rho_D, p_D)$ be the unique weak solution to \eqref{eq:model_eqnts} in the sense of Definition~\ref{eq:weak_solutions} with initial values ($n_0, \rho_0$). Let $(n, \rho, p)$ be the unique weak solution to \eqref{eq:model_eqnts} with $D=0$ in the sense of Definition~\ref{eq:weak_solutions} with same initial values. Suppose that $B_{R}(x_0) \subset \{ \rho(t_1) = 0\}$ for some $R \le 1$ and some $t_1 \in (0, T]$. Then there exists $ \delta_{0} = \delta_0(R, d, \Vert n_0 \Vert_{L^{\infty}(\mathbf{R}^d)})$ and a constant $c = c({d, \Vert n_0 \Vert_{L^{\infty}(\mathbf{R}^d)}})$ such that the following holds:

\medskip

For every $\delta \le \delta_0$ there exists $D_0=D_0(\delta, T) > 0$ such that if $D<D_0$ then
\begin{equation}
B_{R}(x_0) \cap \{p_D(t_0) > 0\} \subset A_{\frac{R}{2}, R}(x_0) \,\hbox{ implies } \,B_{R}(x_0) \cap \{p_D(t_1) > 0\} \subset A_{\frac{R}{4}, R}(x_0),
\end{equation}
 as long as $t_0$ satisfies  
 \begin{equation*}
 0<t_1-t_0 \le \frac{R^k}{8c\delta},\quad k={4+\frac{d}{d+1}+\frac{d}{2}}.
 \end{equation*}
\end{proposition}
\begin{proof}
Let $\delta_0$ to be fixed later, let $\delta \le \delta_0$. Let $D_0 > 0$ to be fixed later. Let $D \le D_0$ and assume that $B_{R}(x_0) \cap \{p_D(t_0) > 0\} \subset A_{\frac{R}{2}, R}(x_0)$. Let $r: [t_0, t_1] \to [\frac{R}{4}, \frac{R}{2}]$ a non-increasing function to be defined later. We construct a barrier (see Figure \ref{fig:barr}) $\phi: [t_0, t_1] \times B_R(x_0) \to \mathbf{R}$ by requiring that for every $t \in [t_0, t_1]$
\begin{equation}
\begin{cases}
-\Delta \phi(t) = \rho_D(t)n_D(t)\quad &\text{on}\ A_{r(t), R}(x_0),
\\ \phi(t) = p_D(t)\quad &\text{on}\ \partial B_{R}(x_0),
\\ \phi(t) = 0\quad &\text{on}\ \partial B_{r(t)}(x_0).
\end{cases}
\end{equation}

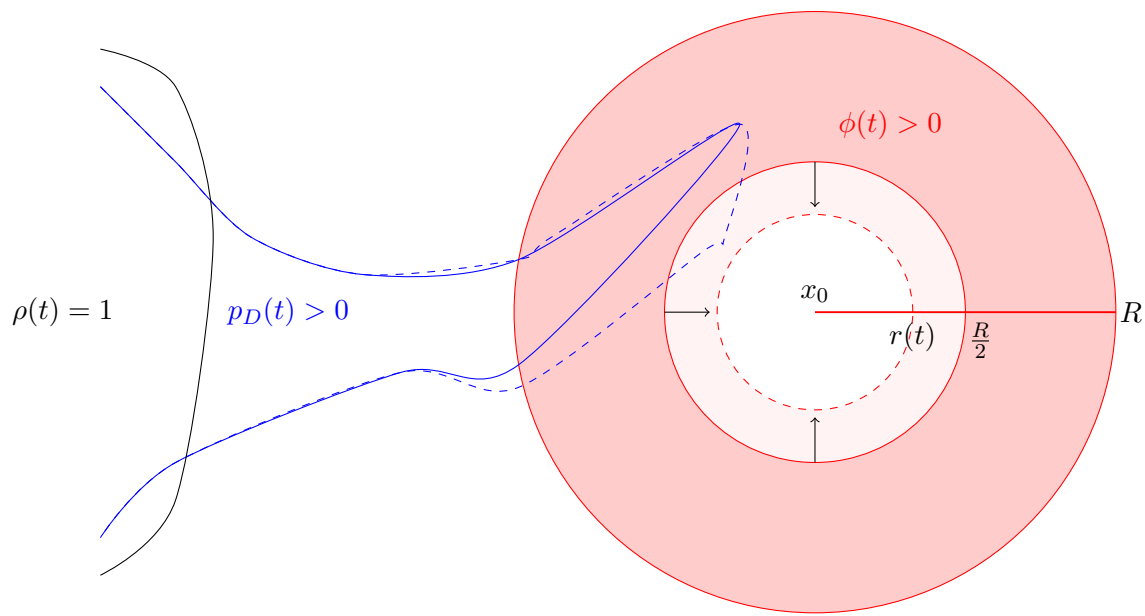
\begin{figure}
\begin{tikzpicture}
    \fill [red,even odd rule, opacity = 0.2] (12,0) circle[radius=2cm] circle[radius=4cm];
    \fill [red,even odd rule, opacity = 0.05] (12,0) circle[radius=1.3cm] circle[radius=2cm];
    \draw [red] (12,0) circle[radius = 2cm];
    \draw [red] (12,0) circle[radius = 4cm];
    \draw [red, dashed] (12,0) circle[radius = 1.3cm];

    \draw [line width=0.25mm, red ] (12,0) -- (16,0);
    \node at (16.2, 0) {$R$};
    \node at (14.2,0) [below] {$\frac{R}{2}$};
    \node at (13.3, 0) [below] {$r(t)$};

    \node[blue] at (5,0) {$p_D(t) > 0$};
    \node[red] at (13, 2.5) {$\phi(t) > 0$};

    \draw [blue] plot [smooth] coordinates { (2.5, 3)  (3.5, 2) (4.5, 1) (6, 0.5) (8.06, 0.7) (11,2.5) (8.06, -0.7) (6.5, -0.8) (3.5, -2) (2.5, -3) };
    \draw [blue, dashed] plot [smooth] coordinates { (2.5, 3)  (3.5, 2) (4.5, 1) (6, 0.5) (8.06, 0.7) (8.5, 1) (11,2.5) (10.8, 1) (10.5, 0.8) (8.06, -1) (6.5, -0.8) (3.5, -2) (2.5, -3) };

    \draw [black] plot [smooth] coordinates { (2.5, 3.5)  (3.5, 3) (4, 1) (3.5, -2.5) (2.5, -3.5) };

    \node[black] at (2,0) {$\rho(t) = 1$};
    \node[black] at (12, 0) [above] {$x_0$};

    \draw[->] (12,2) -- (12, 1.4);
    \draw[->] (10,0) -- (10.6, 0);
    \draw[->] (12,-2) -- (12, -1.4);
\end{tikzpicture}
\caption{Illustration of the barrier construction. The darker annulus represents $A_{\frac{R}{2}, R}(x_0)$. The continuous line entering it represents the zone where the pressure $p_D$ is positive at the initial time $t_0$. As time goes on, this region expands but stays in the annulus delimited by the outermost circle and the dashed inner circle of radius $r(t) \ge \frac{R}{4}$.}\label{fig:barr}
\end{figure}

Step 1. We claim that there exists a constant $c = c(d, \Vert n_0 \Vert_{L^\infty(\mathbf{R}^d)})$ depending only on $d$ and $\Vert n_0 \Vert_{L^{\infty}(\mathbf{R}^d)}$ such that for every $t \in [t_0, t_1]$
\begin{equation}\label{eq:a_priori_c1_barrier}
\Vert \phi(t) \Vert_{C^1\big(A_{r(t), \frac{4r(t)}{3}(x_0)}\big)} \le \frac{c_d}{R^{\frac{d}{d+1} + \frac{d}{2} + 3}}\left( \delta + \Vert p(t) - p_D(t) \Vert_{L^2(\mathbf{R}^d)}\right).
\end{equation}
To prove \eqref{eq:a_priori_c1_barrier} we start by observing that by the Sobolev embedding theorem there exists a constant $C_{r(t), R}$ such that
\begin{equation}
\Vert \phi(t) \Vert_{C^1\big(A_{r(t), \frac{4r(t)}{3}}(x_0)\big)} \le C_{r(t), R}\Vert \phi(t) \Vert_{W^{2, d+1}\big(A_{r(t), \frac{4r(t)}{3}}(x_0)\big)}.
\end{equation}
Using a scaling argument, and the bounds $\frac{R}{4} \le r(t)$, $R \le 1$ we get
\begin{equation}
\Vert \phi(t) \Vert_{C^1\big(A_{r(t), \frac{4r(t)}{3}}(x_0)\big)} \le \frac{c_d}{R^{\frac{d}{d+1}+1}}\Vert \phi(t) \Vert_{W^{2, d+1}\big(A_{r(t), \frac{4r(t)}{3}}(x_0)\big)}.
\end{equation}
Using \cite[Theorem 9.13]{Gilbarg2001} with $p = d+1$ and a scaling argument we obtain that
\begin{align*}
\Vert \phi(t) \Vert_{W^{2,d+1}\big(A_{r(t), \frac{4r(t)}{3}}(x_0)\big)} \le \frac{c}{R^2}\bigg( &\Vert \phi(t) \Vert_{L^{d+1}\big(A_{r(t), \frac{5r(t)}{3}}(x_0) \big)} 
\\ &+ \Vert n_D(t)\rho_D(t) \Vert_{L^{d+1}\big(A_{r(t), \frac{5r(t)}{3}}(x_0)\big)} \bigg)
\end{align*}
We now observe that by assumption $B_{R}(x_0) \subset \{ \rho(t_1) = 0\}$. By Item \eqref{item:rho_conv} in Theorem~\ref{thm:conv_D_zero} we can thus select $D_0(\delta) > 0$ such that 
\begin{equation}
\Vert \rho_D(t) \Vert_{L^{d+1}(B_{R}(x_0))} \le \Vert \rho_D(t_1) \Vert_{L^{d+1}(B_{R}(x_0))} \le \delta\quad \forall D \le D_0.
\end{equation}
In particular, using also that $\Vert n_D \Vert_{L^{\infty}(\mathbf{R}^d)} \le \Vert n_0 \Vert_{L^{\infty}(\mathbf{R}^d)}$ we obtain
\begin{equation}
\Vert \phi(t) \Vert_{W^{2,{d+1}}\big(A_{r(t), \frac{4r(t)}{3}}(x_0)\big)} \le \frac{c}{R^2}\left( \Vert \phi(t) \Vert_{L^{d+1}\big(A_{r(t), \frac{5r(t)}{3}}(x_0)\big)} +\Vert n_0 \Vert_{L^{\infty}(\mathbf{R}^d)} \delta \right).
\end{equation}
To estimate $\Vert \phi(t) \Vert_{L^{d+1}(A_{r(t), \frac{5r(t)}{3}}(x_0))}$ we observe that clearly, using also $R \le 1$,
\begin{equation}
\Vert \phi(t) \Vert_{L^{d+1}\big(A_{r(t), \frac{5r(t)}{3}}(x_0)\big)} \le c_d \Vert \phi(t) \Vert_{L^{\infty}\big(A_{r(t), \frac{5r(t)}{3}}(x_0)\big)}
\end{equation}
By using the fact that $\phi(t) = 0$ on $\partial B_{r(t)}(x_0)$, the weak Harnack's inequality for the subsolution $\phi_+$ implies that there exists a constant $c_d$ depending only on the dimension $d$ such that
\begin{align}
\Vert \phi(t) \Vert_{L^{\infty}\big(A_{r(t), \frac{5r(t)}{3}}(x_0)\big)} \le \frac{c_d}{R^{\frac{d}{2}}} \left(\Vert \phi(t) \Vert_{L^2\big(A_{r(t), 2r(t)}(x_0)\big)}+  \Vert n_0 \Vert_{L^{\infty}(\mathbf{R}^d)} \delta\right).
\end{align}
Observe now that by the maximum principle we have $p_D(t) \ge \phi(t)$ on $B_R(x_0)$, in particular, using also the assumption that $p(t) = 0$ on $B_R(x_0)$
\begin{align}
    \Vert \phi(t) \Vert_{L^2\big(A_{r(t), 2r(t)}(x_0)\big)} &\le \Vert p_D(t) \Vert_{L^2\big(A_{\frac{R}{4}, R}(x_0)\big)}
    \\ &= \Vert p(t) - p_D(t) \Vert_{L^2\big(A_{\frac{R}{4}, R}(x_0)\big)}.
\end{align}
Putting things together we thus obtain
\begin{align}
    \Vert \phi(t) \Vert_{C^1\big(A_{r(t), \frac{4r(t)}{3}}(x_0)\big)} & \le \frac{c_d}{R^{\frac{d}{d+1} + \frac{d}{2} + 3}}\left( \delta + \Vert p(t) - p_D(t) \Vert_{L^2(\mathbf{R}^d)}\right),
\end{align}
which is \eqref{eq:a_priori_c1_barrier}.

Step 2. We now choose
\begin{equation}
r(t) = \frac{R}{2} - \int_{t_0}^t \frac{c_d}{R^{\frac{d}{d+1} + \frac{d}{2} + 3}}\left( \delta + \Vert p(s) - p_D(s) \Vert_{L^2(\mathbf{R}^d)}\right)ds,
\end{equation}
and we define, for $s \in [t_0, t_1]$ and $x \in A_{r(t), R}(x_0)$
\begin{equation}
q(t, x) := \int_{t_0}^t \phi_+(s,x)ds.
\end{equation}
We claim that for every $t \in [t_0, t_1]$
\begin{equation}\label{eq:supersolution_toprove}
\chi_{A(t)} - \Delta q(t) \ge \chi_{A(t_0)} + \int_{t_0}^t \rho_D(s) n_D(s)\chi_{A(s)} ds\quad \text{on}\ B_{R}(x_0),
\end{equation}
where we set $A(t) := A_{r(t), R}(x_0)$.

To prove \eqref{eq:supersolution_toprove} we pick $\varphi \in C^{\infty}_c(B_R(x_0)), \varphi \ge 0$ and we compute, using the fact that $\nabla \phi_+(s) = \nabla \phi \mathbf{1}_{\phi(s) > 0}$
\begin{align}
&\begin{aligned}
\int_{B_R(x_0)} \nabla q(t) \cdot \nabla \varphi dx = \int_{t_0}^t \int_{B_{R}(x_0)} \nabla \phi_+(s) \cdot \nabla \varphi dx ds\label{eq:weak_form_supersol}
\end{aligned}
\\ &\begin{aligned}
\phantom{\int_{B_R(x_0)} \nabla q(t) \cdot \nabla \varphi dx} =  \int_{t_0}^t \int_{A_{r(s), R}(x_0)} \nabla \phi(s) \cdot \nabla \varphi dx ds
\end{aligned}
\\ &\begin{aligned}
\phantom{\int_{B_R(x_0)} \nabla q(t) \cdot \nabla \varphi dx} = &\int_{t_0}^t \int_{\partial A_{r(s), R}(x_0)} \nabla \phi(s) \cdot \nu(s) \varphi \mathcal{H}^{d-1}(dx) ds 
\\ &- \int_{t_0}^t \int_{A_{r(s), R}(x_0)} \Delta \phi(s) \varphi dx ds,
\end{aligned}
\end{align}
where we denoted by $\nu(s)$ the outer unit normal to $A_{r(s), R}(x_0)$. We continue by observing that $\nabla \phi(s) \cdot \nu(s) = - |\nabla \phi(s)|$. We also observe that by what we proved in Step 1 and by our choice of $r(t)$ we have $- |\nabla \phi(s)|  \ge \dot{r}(s)$ on $\partial B_{r(s)}(x_0)$. We denote by $V(s, x)$ the normal velocity of a point $x \in \partial A_{r(s), R}(x_0)$. We compute, using also that $\varphi = 0$ on $\partial B_R(x_0)$,
\begin{align*}
\int_{t_0}^t \int_{\partial A_{r(s), R}(x_0)} \nabla \phi(s) \cdot \nu(s) \varphi \mathcal{H}^{d-1}(dx) ds &= -\int_{t_0}^t \int_{\partial A_{r(s), R}(x_0)} |\nabla \phi(s)|\varphi \mathcal{H}^{d-1}(dx) ds
\\ & \ge \int_{t_0}^t \int_{\partial A_{r(s), R}(x_0)} \dot{r}(s) \varphi dx ds
\\ & = -\int_{t_0}^t \int_{\partial A_{r(s), R}(x_0)} V(s, x) \varphi(x) dx ds
\\ & = -\int_{t_0}^t \frac{d}{ds} \int_{A_{r(s), R}(x_0)} \varphi(x) dx ds
\\ & = \int_{B_R(x_0)} \chi_{A(t_0)} \varphi dx - \int_{B_R(x_0)} \chi_{A(t)} \varphi dx.
\end{align*}
Inserting back into \eqref{eq:weak_form_supersol} we obtain
\begin{align*}
\int_{B_R(x_0)} \nabla q(t) \cdot \nabla \varphi dx \ge &\int_{B_R(x_0)} \int_{t_0}^t \rho_D(s) n_D(s) \chi_{A(s)}ds \varphi dx
\\ & + \int_{B_R(x_0)} \chi_{A(t_0)} \varphi dx - \int_{B_R(x_0)} \chi_{A(t)} \varphi dx,
\end{align*}
which is the weak formulation for \eqref{eq:supersolution_toprove}.

Step 3. We now define 
\begin{equation}
v_D(t, x) := \int_{t_0}^t p_D(s, x)ds,
\end{equation}
and we observe that
\begin{equation}\label{eq:exact_eq_for_rho_D}
\rho_D(t) - \Delta v_D(t) = \rho_D(t_0) + \int_{t_0}^t \rho_D(s) n_D(s) ds.
\end{equation}
We claim that 
\begin{equation}\label{eq:comp}
v_D(t, x) \le q(t, x)\quad \text{for}\ t \in [t_0, t_1], x \in B_R(x_0).
\end{equation}
This of course implies that for every $t \in [t_0, t_1]$ we have $\{p_D(t) > 0\} \cap B_R(x_0) \subset A_{r(t), R}(x_0)$.

To prove \eqref{eq:comp} we subtract \eqref{eq:supersolution_toprove} from \eqref{eq:exact_eq_for_rho_D} and we test the resulting inequality with $(v_D(t) - q(t))_{+}$ (which is an admissible test function because $v_D(t) = q(t)$ on $\partial B_{R}(x_0)$). We obtain, using also that at the initial time $t_0$ we have by assumption that $\chi_{A(t_0)} \ge \rho_D(t_0)$ in the ball $B_R(x_0)$,
\begin{align}
&\int_{B_{R}(x_0)} (\rho_D(t)-\chi_{A(t)})(v_D(t) - q(t))_{+} dx + \int_{B_{R}(x_0)}|\nabla (v_D(t) - q(t))_{+}|^2 ds\label{eq:comparison}
\\ & \le \int_{B_{R}(x_0)} (\rho_D(t_0)-\chi_{A(t_0)}(x_0))(v_D(t) - q(t))_{+} dx
\\ & \le 0.
\end{align}
Observe that both terms on the left-hand side of \eqref{eq:comparison} are non-negative. Indeed, if for some $x \in B_R(x_0)$ we have $v_D(t, x) > q(t, x)$, then clearly $\rho_D(t, x) = 1$. We thus infer that there exists $c(t) \in [0, +\infty)$ such that
\begin{equation}
(v_D(t) - q(t))_+ = c(t)\quad \text{on}\ B_R(x_0),
\end{equation}
and since $v_D(t) = q(t)$ on $\partial B_R(x_0)$ we have $c(t) = 0$. In other words \eqref{eq:comp} holds.

Step 4. Conclusion.
\\
Because of Item \eqref{item:pressure_conver} in Remark~\ref{rem:density_smooth} we can assume that $D_0$ is so small that for every $D \le D_0$
\begin{equation}
\Vert p - p_D \Vert_{L^2((0, 2t_1) \times \mathbf{R}^d)} \le \delta.
\end{equation}
In particular, this yields
\begin{equation}
r(t_1) \ge \frac{R}{2} - \frac{c}{R^{\frac{d}{d+1}+\frac{d}{2}+3}}\left( \delta (t_1 - t_0) + \delta\right).
\end{equation}
We need $r(t_1) \ge \frac{R}{4}$, which is satisfied provided
\begin{equation}\label{eq:cond_delta}
\frac{R}{4} \ge \frac{c}{R^{\frac{d}{d+1}+\frac{d}{2}+3}}\left( \delta (t_1 - t_0) + \delta\right).
\end{equation}
We define 
\begin{equation}
\delta_0 := \frac{R^{4 + \frac{d}{2} + \frac{d}{d+1}}}{8c},
\end{equation}
then for $\delta \le \delta_0$ we have that \eqref{eq:cond_delta} is satisfied provided
\begin{equation}
\frac{R^{4+\frac{d}{d+1} + \frac{d}{2}}}{8c\delta} \ge t_1 - t_0.\qedhere
\end{equation}
%
\end{proof}
\begin{proof}[Proof of Theorem~\ref{thm:conv_patches}]
We will prove the following: for every $t\in (\hat{t}, T)$, given $\epsilon>0$ there exists $\overline{D} > 0$ such that for every $D \le \overline{D}$
\begin{equation}
d_H(\Gamma_D(t), \Gamma(t)) \le \epsilon.
\end{equation}
Because of Proposition~\ref{prop:easy_part}, it suffices to show that there exists $\overline{D} > 0$ such that for every $D \le \overline{D}$
\begin{equation}\label{eq:to_show}
\sup_{x \in \Gamma_D(t)} d(x, \Gamma(t)) \le \epsilon.
\end{equation}
To prove \eqref{eq:to_show} we proceed in three steps.

Step 1. Exterior fingers control. Define $\Omega_t = \{ \rho(t) = 1\}$. We claim that we can select $D_1 > 0$ such that for $D \le D_1$
\begin{equation}\label{eq:exterior_fing_cont}
\Omega^{\epsilon}_t := \left\{ x \in \mathbf{R}^d: d(x, \Omega_t) < \epsilon\right\} \supset \Gamma_D(t) \setminus \{\rho(t) = 1\}..
\end{equation}
To prove \eqref{eq:exterior_fing_cont}, we can apply Proposition~\ref{prop:ext_cont} with $R = \frac{\epsilon}{2}$, $t_1 = t$, $t_0 = 0$ and $\delta$ small enough to find that there exists $D_1 > 0$ such that for every  $x_0 \in \partial \Omega_t^{\frac{\epsilon}{2}}$ and for every $D \le D_1$
\begin{equation}
B_{\frac{\epsilon}{2}}(x_0) \cap \{p_D(t) > 0\} \subset A_{\frac{\epsilon}{8}, \frac{\epsilon}{2}}(x_0).
\end{equation}
In particular we infer $\Gamma_D(t)\setminus \{\rho(t) = 1\} \subset \Omega_t^\epsilon$.

Step 2. Interior fingers control. 
\\
By the smoothness of $p$ on $[\hat{t}, T]$ we can find a constant $C > 0$ such that $|\nabla p(s, x)| \le C$ for every $s \in [\hat{t}, T]$ and every $x \in \partial \{ \rho(s) = 1 \}$. We let $\gamma = \min\left( 1, \frac{\epsilon}{2C}, t-\hat{t} \right)$. We apply Proposition~\ref{prop:inner_control} with $\delta \le \frac{\gamma^2}{5}$. Then there exists $D_2 > 0$ such that for $D \le D_2$
\begin{equation}
U_{\gamma^2}(t-\gamma) \subset \{p_D(t) > 0\}.
\end{equation}
We now take $x \in \Gamma_D(t) \cap \{\rho(t) = 1\}$. By what we just showed we have $x \in \mathbf{R}^d \setminus U_{\gamma^2}(t-\gamma)$. In particular there exists $y \in \partial \Gamma(t-\gamma)$ such that $|x - y| < \gamma^2$. Let now $y:[t-\gamma, t] \to \mathbf{R}^d$ be the curve defined by
\begin{equation}
\begin{cases}
\dot{y}(s) = -\nabla p(s, y(s))\quad &\text{for}\ s \in [t-\gamma, t]
\\ y(t-\gamma) = y.
\end{cases}
\end{equation}
Then $y(t) \in \Gamma(t)$ and
\begin{align*}
d(x, \Gamma(t)) &\le |x-y(t)|
\\ &\le |x - y(t-\gamma)| + |y(t-\gamma) - y(t)|
\\ &\le \gamma^2 + \int_{t-\gamma}^t |\nabla p(s, y(s))|ds \le \gamma^2 + C\gamma \le \epsilon.
\end{align*}
Step 3. Conclusion.
\\
Take $\overline{D} = \min(D_1, D_2)$ then for any $D \le \overline{D}$ and any $x \in \Gamma_D(t)$ we have
\begin{enumerate}
\item If $x \in \mathbf{R}^d\setminus\{\rho(t) = 1\}$ then, by what we proved in Step 1 we have $x \in \Omega_t^{\epsilon}$, thus $d(x, \Gamma(t)) \le \epsilon$.
\item Otherwise we apply what we proved in Step 2 to infer that $d(x, \Gamma(t)) \le \epsilon$.
\end{enumerate}
We thus have that
\begin{equation}
\sup_{x \in \Gamma_D(t)} d(x, \Gamma(t)) \le \epsilon,
\end{equation}
and the proof of Theorem~\ref{thm:conv_patches} is concluded.
\end{proof}

\section*{Acknowledgements}

The first author is partially supported by NSF grant DMS-1900804. The second author has received funding from the Deutsche Forschungsgemeinschaft (DFG, German Research Foundation) under Germany's Excellence Strategy -- EXC-2047/1 -- 390685813. This work was completed during an extended research stay of the second author at UCLA in the framework of the BIGS Global Math Network program. The second author thanks BIGS and UCLA for the support and for providing such a stimulating working environment.

\bibliography{biblio_kim_lelmi}{}
\bibliographystyle{plain}

\end{document}